\documentclass{amsart}

\usepackage{amsmath,amsthm,amsfonts,amssymb,amscd,verbatim,latexsym}
\usepackage{hyperref}
\usepackage{xcolor}
\hypersetup{
    colorlinks,
    linkcolor={blue!80!black},
    citecolor={blue!80!black},
}
\usepackage{multicol}
\usepackage{enumerate}
\usepackage{dsfont}
\usepackage{array}
\usepackage{color}
\usepackage{bbold}
\usepackage{bbm}

\newcolumntype{L}[1]{>{\raggedright\let\newline\\\arraybackslash\hspace{0pt}}p{#1}}
\newcolumntype{C}[1]{>{\centering\let\newline\\\arraybackslash\hspace{0pt}}p{#1}}
\newcolumntype{R}[1]{>{\raggedleft\let\newline\\\arraybackslash\hspace{0pt}}p{#1}}

\newcommand{\cA}{\mathcal{A}}
\newcommand{\cB}{\mathcal{B}}
\newcommand{\bG}{\mathbf{G}}
\DeclareMathOperator{\Id}{Id}
\DeclareMathOperator{\SL}{SL}
\DeclareMathOperator{\Aut}{Aut}
\DeclareMathOperator{\gldim}{gldim}
\DeclareMathOperator{\GKdim}{GKdim}
\newcommand{\kk}{\Bbbk}
\newcommand{\inv}{^{-1}}
\renewcommand{\i}{\mathbbm{i}}

\newcommand{\kashina}[1]{Kashina \#{#1}}
\newcommand{\twokashinas}[2]{Kashinas \#{#1} and \#{#2}}
\newcommand{\threekashinas}[3]{Kashinas \#{#1}, \#{#2}, and \#{#3}}

\newcommand\grp[1]{\left\langle #1 \right\rangle}

\newtheorem{theorem}{Theorem}[section]
\newtheorem{lemma}[theorem]{Lemma}

\theoremstyle{definition}
\newtheorem{definition}[theorem]{Definition}

\newtheorem{notation}[theorem]{Notation}
\newtheorem{Conjecture}[theorem]{Conjecture}

\theoremstyle{remark}
\newtheorem{remark}[theorem]{Remark}

\numberwithin{equation}{section}

\title[Semisimple reflection Hopf algebras of dimension sixteen]{Semisimple reflection Hopf algebras\\ of dimension sixteen}
\author{Luigi Ferraro}
\address{Texas Tech University, Department of Mathematics and Statistics, Lubbock, Texas 79409}
\email{lferraro@ttu.edu}

\author{Ellen Kirkman}
\address{Wake Forest University, Department of Mathematics and Statistics, P. O. Box 7388, Winston-Salem, North Carolina 27109} 
\email{kirkman@wfu.edu}

\author{W. Frank Moore}
\address{Wake Forest University, Department of Mathematics and Statistics, P. O. Box 7388, Winston-Salem, North Carolina 27109}
\email{moorewf@wfu.edu}

\author{Robert Won}
\address{University of Washington, Department of Mathematics, Box 354350, Seattle, Washington 98195}
\email{robwon@uw.edu, robwon@gmail.com}


\begin{document}

\maketitle
\begin{abstract}
For each nontrivial semisimple Hopf algebra $H$ of dimension sixteen over $\mathbb{C}$,  the smallest dimension inner-faithful representation of $H$ acting on a quadratic  AS regular algebra $A$ of dimension 2 or 3, homogeneously and preserving the grading, is determined. Each invariant subring $A^H$ is determined.  When $A^H$ is also AS regular, thus providing a generalization of the Chevalley--Shephard--Todd Theorem, we say that $H$ is a reflection Hopf algebra for $A$.
\end{abstract}
\section{Introduction}

Throughout let $\kk = \mathbb{C}$, and denote the square root of $-1$ by $\i$.  A finite subgroup $G$ of GL$_n(\kk)$, acting linearly as graded automorphisms on a (commutative) polynomial ring $A = \kk[x_1, \dots, x_n]$, is called a {\em reflection group} if the group is generated by  elements $g \in G$, which act on the vector space $\bigoplus \kk x_i$ with fixed subspace of codimension 1; this condition is equivalent to the condition that all the eigenvalues of  $g$ are $1$, with the exception of one eigenvalue that is a root of unity (sometimes such elements $g$ are called {\em pseudoreflections} when the exceptional eigenvalue is a root of unity that is not $-1$).  Chevalley \cite{Chev} and Shephard and Todd \cite{SheTod} showed that over a field of characteristic zero, a group $G$  is a reflection group if and only if the invariant subalgebra $A^G$ is a polynomial ring, and Shephard and Todd  \cite{SheTod} presented a complete classification of the reflection groups into three infinite families and 34 exceptional groups.  Reflection groups have played an important role in many contexts, including in representation theory, combinatorics, commutative ring theory, and algebraic geometry.

There has been interest in extending the Chevalley--Shephard--Todd theorem to a noncommutative context (replacing the commutative polynomial ring with a noncommutative algebra $A$), and in \cite[Definition 2.2]{KKZ1} an analog of a reflection (called a {\em quasi-reflection} in that paper) was defined for a graded automorphism $g$ of an Artin--Schelter regular (AS regular) algebra $A$ that is generated in degree 1 (Definition~\ref{def:asregular}).  When such an AS regular algebra $A$  is commutative, it is isomorphic to a commutative polynomial ring, so this particular noncommutative setting generalizes the classical commutative polynomial algebra case.  Moreover, examples suggest that the proper analog of a reflection group for $A$ is a group $G$ such that the invariant subalgebra $A^G$ is also AS regular. An extended notion of ``reflection" introduced in \cite{KKZ1}  (that involves ``trace functions'' rather than eigenvalues)  was used in  \cite{KKZ2} to prove a version of the Chevalley--Shephard--Todd Theorem for groups acting on skew polynomial rings (a second proof was given in \cite{BB}).  

 To extend classical invariant theory further, to a noncocommutative context, the group $G$ can be replaced by a semisimple Hopf algebra $H$ (see  \cite{KKZ3}). Several results for the action of a finite subgroup of $\SL_{2}(\kk)$ on $\kk[u,v]$ have been extended to this context (e.g. \cite{CKWZ, CKWZ1, CKWZ2}). However, it has appeared more difficult to extend the Chevalley--Shephard--Todd Theorem to this Hopf setting.  To this end we consider a pair $(A,H)$, where $A$ is an AS regular algebra and $H$ is a (finite-dimensional) semisimple Hopf algebra, equipped with an action of $H$ on $A$ that is homogeneous, preserves the grading, and is inner-faithful on $A$ (meaning that no non-zero Hopf ideal of $H$ annihilates $A$). We call $H$ a {\em reflection Hopf algebra} for $A$ (\cite[Definition 3.2]{KKZ6}) when the ring of invariants $A^H$ is AS regular.  It is not yet clear what conditions on the pair $(A,H)$ and the action force $A^H$ to be AS regular. In \cite[Examples 7.4 and 7.6]{KKZ3}, it was shown that the Kac-Paljutkin algebra \cite{KP} is a reflection Hopf algebra for both $A = \kk_{-1}[u,v]$ and $A = \kk_{\i}[u,v]$.  In \cite{KKZ6}, actions of duals of group algebras $H = \kk G^\circ$ (or equivalently, group coactions) were considered, and some {\em dual reflection groups} were constructed. In \cite{FKMW}, three infinite families of Hopf algebras were shown to be reflection Hopf algebras. The goal of this paper is to provide further data toward a better understanding of reflection Hopf algebras.  

 The Kac-Paljutkin  algebra of dimension 8 is the smallest-dimensional nontrivial semisimple Hopf algebra (in the sense that it is not isomorphic to a group algebra or its dual). For primes $p$ and $q$ it is known that semisimple Hopf algebras of dimension $p$ \cite{Zhu}, $2p$ \cite{M2}, $p^2$ \cite{M3}, and $pq$ \cite{EG, GW} are trivial. The families considered in \cite{FKMW} include the two nontrivial semisimple Hopf algebras of dimension twelve.
 Hence the next dimension to consider is sixteen.
 Y. Kashina \cite[Table 1] {kashina} presented an explicit description of the sixteen nontrivial semisimple Hopf algebras of dimension sixteen, and we follow her classification. In particular, we refer to each of these Hopf algebras by the number given in \cite{kashina}.
 For each of these sixteen Hopf algebras $H$, we first classify the inner-faithful representations of $H$ using Theorem~\ref{innerfaithful}; the inner-faithful condition is included so that the action does not factor through a Hopf algebra of smaller dimension (which might be a group algebra). The table below summarizes the inner-faithful representations of smallest dimension for each the sixteen-dimensional Hopf algebras.\\

\begin{center}
 \begin{tabular}{|>{\centering\arraybackslash}m{1.6cm}|>{\centering\arraybackslash}m{2.7cm}|c|c|c|}
\hline
Hopf algebra & Lowest dimension inner-faithful representation & Theorem\\
\hline \hline 1, 3 & $3$ & Theorem~\ref{innerfaithful13}\\
\hline
2, 4 & $3$ & Theorem~\ref{innerfaithful24}\\
\hline 5, 7   & $2$  & Theorem~\ref{fixedrings57} \\
\hline 6, 10, 11  &   $3$  &  Theorem~\ref{fixedring61011}\\
\hline 8, 9 &   $3$  & Theorem~\ref{fixedrings89} \\
\hline 12 &  $2$  & Theorem~\ref{fixedring12}  \\
\hline 13  &  $2$  & Theorem~\ref{fixedring13}   \\
\hline 14  &   $3$  &  Theorem~\ref{fixedring14}\\
\hline 15  &  $3$  & Theorem~\ref{fixedring15} \\
\hline 16   & $2$  &  Theorem~\ref{fixedring16}\\
\hline
\end{tabular}
\end{center} 
~\\
~\\
 For each sixteen-dimensional semisimple Hopf algebra $H$ there is  an AS regular algebra $A$ of dimension 2 or 3 on which it acts inner-faithfully.
After finding the inner-faithful representations of $H$, we determine  a family  of the smallest-dimensional quadratic AS regular algebras $A$ on which $H$ acts inner-faithfully, compute the invariant subrings $A^H$, and determine which are AS regular.
The table below provides references for the invariant subrings, and lists the Hopf algebras that are reflection Hopf algebras for some lowest dimension inner-faithful representation.

\begin{center}
 \begin{tabular}{|>{\centering\arraybackslash}m{1.6cm}|>{\centering\arraybackslash}m{2.7cm}|c|c|c|}
\hline
Hopf algebra & Reflection representation & Theorem\\
\hline
1,2 & Yes & Theorem~\ref{fixedrings1234}\\
\hline
3,4 & No & Theorem~\ref{fixedrings1234}\\
\hline
5 & Yes & Theorem~\ref{fixedrings57}\\
\hline
6 & Yes & Theorem~\ref{fixedring61011}\\
\hline
7 & No & Theorem~\ref{fixedrings57}\\
\hline
8, 9 & No & Theorem~\ref{fixedrings89}\\
\hline
10, 11  & No & Theorem~\ref{fixedring61011}\\
\hline
12 & Yes & Theorem~\ref{fixedring12}\\
\hline
13 & No & Theorem~\ref{fixedring13}\\
\hline
14 & Yes & Theorem~\ref{fixedring14}\\
\hline
15 & Yes & Theorem~\ref{fixedring15}\\
\hline
16 & Yes & Theorem~\ref{fixedring16}\\
\hline
\end{tabular}
\end{center}
~\\
~\\
For each of the AS regular invariant subrings in the table above, the product of the degrees of the minimal generators of the invariants is equal to 16.  This observation leads to the following conjecture, which holds for group actions on commutative polynomial rings, as well as all Hopf actions that we have computed in the non(co)commutative setting:

 \begin{Conjecture}
When a semisimple Hopf algebra $H$ acts on an AS regular algebra $A$ and the invariant subring $A^H$ is AS regular, there is a minimal generating set of the algebra $A^H$ that has the property that the product of the degrees of the minimal generators is less than or equal to the dimension of $H$ over $\Bbbk$.  An example \cite[Example 1.2(2)]{KWZ}) shows that the product of the degrees is not always equal to the dimension of $H$ over $\Bbbk$.
\end{Conjecture}

The paper is organized as follows.  Background material is presented in Section~\ref{background}, and the subsequent sections consider the Hopf algebras $H$ of dimension sixteen, organized by their groups of grouplike elements $\bG(H)$. In Section~\ref{c2xc2xc2} we consider the case $\bG(H) = C_2 \times  C_2 \times C_2$ (\kashina 1--4), in Section~\ref{c4xc2} we consider the case $\bG(H) = C_4 \times  C_2 $ (\kashina  5--11), in Section~\ref{d8} we consider the case $\bG(H) = D_8$ (\kashina 12 and 13), and in Section~\ref{c2Xc2} we consider the case $\bG(H) = C_2 \times C_2$ (\kashina  14--16). We remark that \kashina 14 and \kashina 16 occur as members of the infinite families considered in \cite{FKMW}; we summarize the results here for the sake of completeness.

Many computations in this paper were aided by the use of the NCAlgebra package in the computer algebra
system Macaulay2 \cite{M2}.\\
~\\
\noindent {\bf Acknowledgment:} We appreciate the careful reading of the manuscript by the referee who caught several mistakes and suggested a simplification of our arguments for inner-faithfulness, which we have adopted.

\section{Background}\label{background}
This section provides notation and background results that will be used throughout. 
Throughout, we follow the standard notation ($\Delta, \epsilon, S$) for the comultiplication, counit, and antipode of a Hopf algebra $H$, and refer the reader to \cite{Mbook} for any undefined terminology concerning Hopf algebras.  Further, we follow  the notation and numbering used in Kashina's classification of the semisimple Hopf algebras of dimension sixteen that are neither commutative nor cocommutative \cite[Table 1]{kashina}. 
\subsection{AS Regular Algebras}

All sixteen of the dimension sixteen Hopf algebras act on Artin--Schelter regular algebras, which are algebras possessing the homological properties of commutative polynomial rings.
\begin{definition}
\label{def:asregular}
Let $A$ be a connected graded algebra. Then $A$ is
{\it Artin--Schelter (AS) regular} if
\begin{enumerate}
\item[(1)] $A$ has finite global dimension,
\item[(2)] $A$ has finite Gelfand-Kirillov dimension, and
\item[(3)] $A$ is {\it Artin--Schelter Gorenstein}, i.e.,
\begin{enumerate}
\item[(a)] $A$ has finite injective dimension $d < \infty$, and
\item[(b)] Ext$^i_A(\kk,A) = \delta_{i,d} \cdot \kk(l)$ for some
$l \in \mathbb{Z}$.
\end{enumerate}
\end{enumerate}
\end{definition}

Examples of AS regular algebras include skew polynomial rings and graded Ore extensions of AS regular algebras.  We will use the following well-known fact (see e.g. \cite[Lemma 1.2]{CKZ}) to show that an invariant subring is not AS regular.

\begin{lemma} \label{notregular}
If $A$ is an AS regular algebra of GK dimension 2 (resp. 3), then $A$ is generated by 2 elements (resp. 2 or 3 elements).
\end{lemma}

We will encounter algebras of the form $\kk\langle u,v \rangle/(u^2 - cv^2)$ for $c \in \kk^\times$ several times 
in this paper, so we record here the following lemma which identifies a set of monomials which forms a $\kk$-basis.
\begin{lemma} \textup{(\cite[Lemma 1.3]{FKMW})} \label{lem:u2v2Basis}
Let $c \in \kk^\times$ and
$A = \kk\langle u,v \rangle/(u^2 - cv^2)$.  Then $A$ is AS regular, and 
the monomials $\{u^i(vu)^jv^\ell\}$ with $i,j$ nonnegative integers
and $\ell \in \{0,1\}$ form a $\kk$-basis of $A$.
\end{lemma}

\subsection{Inner-faithful Actions}

A module $V$ over a Hopf algebra $H$ is \emph{inner-faithful} if the only
Hopf ideal which annihilates $V$ is the zero ideal.  We record the following
result which is due to Brauer \cite{Brau}, Burnside \cite{Burn} and Steinberg \cite{Stein}
in the case of a group algebra of a finite group, and due to  Rieffel \cite{R} (see also  Passman and Quinn \cite{PassQui})
in the case of a finite-dimensional semisimple Hopf algebra.  

\begin{theorem} \textup{(\cite[Theorem 1.4]{FKMW})} \label{innerfaithful}
Let $V$ be a module over a finite-dimensional semisimple Hopf algebra $H$.
Then the following conditions are equivalent.
\begin{enumerate}
\item $V$ is an inner-faithful $H$-module,
\item The tensor algebra $T(V)$ is a faithful $H$-module,
\item Every simple $H$-module appears as a direct summand of $V^{\otimes n}$
for some $n$.
\end{enumerate}
\end{theorem}
 The following lemma is used to prove that actions are not inner-faithful.
\begin{lemma}\textup{(\cite[Lemma 1.7]{FKMW})} \label{HopfIdeal}
Let $H$ be a Hopf algebra and $g\in\bG(H)$. Then the two-sided ideal generated by $1-g$ is a Hopf ideal. 
\end{lemma}
\subsection{The Grothendieck Ring}

To determine if a particular representation $V$ of a semisimple Hopf algebra $H$ is inner-faithful, by Theorem~\ref{innerfaithful} one can compute the decomposition of tensor powers of $V$ into irreducible $H$-modules, i.e., to perform computations in the Grothendieck ring $K_0(H)$. 

As noted in \cite{kashina}, the nontrivial Hopf algebras of dimension sixteen have two possible Wedderburn decompositions:
\[\kk^{\oplus 8} \oplus M_2(\kk) \oplus M_2(\kk)\]
or 
\[\kk^{\oplus 4} \oplus M_2(\kk) \oplus M_2(\kk) \oplus M_2(\kk)\]
and hence $H$ has either two non-isomorphic irreducible two-dimensional $H$-modules and eight one-dimensional $H$-modules, or three non-isomorphic irreducible two-dimensional $H$-modules and four one-dimensional $H$-modules. In what follows we shall use the notation $\pi(u,v)$ for a two-dimensional irreducible $H$-module with vector space basis $u,v$, and $T(t)$ for a one-dimensional irreducible $H$-module with vector space basis $t$. 

In this paper, the graded quadratic AS regular algebras $A$ on which $H$ acts will be generated as algebras by either two or three elements.
In order for $H$ to act inner-faithfully on $A$, the vector space $A_1$ of degree 1 elements of $A$ must be isomorphic, as an $H$-module, to a two-dimensional irreducible representation $A_1 = \kk u \oplus \kk v = \pi(u,v)$ or to the sum of an irreducible two-dimensional and a one-dimensional representation $A_1 = \kk u \oplus \kk v \oplus \kk t = \pi(u,v) \oplus T(t)$. 

We record here the following lemma which identifies inner-faithful actions in many of the cases
that we will consider.

\begin{theorem} \label{dim2and3InnerFaithful}
Let $H$ be a sixteen-dimensional semisimple Hopf algebra such that $|\bG(H^*)| = 8$.  Let $\pi$ be
an irreducible two-dimensional representation.  Then:
\begin{enumerate}
\item The action of $H$ on $\pi$ is inner-faithful if and only if $\pi^* \not\cong \pi$.
\item Suppose that $\pi^* \cong \pi$, and that $\varphi$ is a one-dimensional representation.
Then the action of $H$ on $\pi \oplus \varphi$ is inner-faithful if and only if $\varphi$
is not a direct summand  $\pi\otimes\pi$.
\end{enumerate}
\end{theorem}

\begin{proof}
In this situation, the calculations appearing in \cite[pp. 652--654]{kashina} show that
$\pi^2$ is a multiplicity-free sum of four one-dimensional representations.

Suppose that $\pi^* \cong \pi$.  Then by \cite[Theorem 10.1]{NR},
$\pi \cong \pi\otimes\varphi$ for some one-dimensional representation
$\varphi$.  It follows from Theorem~\ref{innerfaithful} that the
action of $H$ on $\pi$ is not inner-faithful.  Conversely, suppose
that $\pi^* \not\cong \pi$.  Then by \cite[Theorem 10.1]{NR}, for any
one-dimensional representation $\varphi$ appearing in
$\pi \otimes \pi$, one has $\varphi \otimes \pi \cong \pi^*$.
Further, every one-dimensional representation of $H$ appears in either
$\pi\otimes\pi$ or $\pi\otimes\pi^*$, justifying the claim.

Next, suppose that $\pi \cong \pi^*$, and that $\varphi$ is a
one-dimensional representation.  If $\varphi$ appears in
$\pi \otimes \pi$, then clearly the action of $H$ on
$\pi \oplus \varphi$ is not inner-faithful, as the action of $H$ on
$\pi$ is not inner-faithful.  Conversely, suppose that $\varphi$ does
not appear in $\pi \otimes \pi$.  Then \cite[Theorem 10.2]{NR} shows
that the one-dimensional representations that appear in
$\pi \otimes \pi$ form a subgroup of $\bG(H^*)$ and are precisely those
representations $\psi$ that satisfy $\psi\otimes\pi = \pi$.  It
follows that if $\varphi$ does not appear in $\pi \otimes \pi$, then
$\varphi\otimes\pi \cong \pi^*$ and that all of $\bG(H^*)$ is generated
by the aforementioned subgroup and $\varphi$.  It follows that the
action of $H$ on $\pi \oplus \varphi$ is inner-faithful.
\end{proof}

\subsection{Dual Cocycle Twists}
The Hopf algebras \kashina 1--16 are not twists of each other, but several are twists of other algebras. Here we recall facts about twists using dual cocyles.

Let $H$ be a Hopf algebra. An invertible element $\Omega = \sum \Omega^{(1)} \otimes \Omega^{(2)} \in H \otimes H$ is called a \emph{dual $2$-cocycle} if
\[ [(\Delta \otimes \Id)(\Omega)](\Omega \otimes 1) = [(\Id \otimes \Delta )(\Omega)](1 \otimes \Omega).
\]
The dual 2-cocycle $\Omega$ is said to be \emph{normal} if also
\[ (\Id \otimes \epsilon)(\Omega) = (\epsilon \otimes \Id)(\Omega) = 1.
\]
This is the formal dual to the definition of a normal cocycle.

Given a Hopf algebra $H$ and a dual cocycle $\Omega \in H \otimes H$, we can form the \emph{cotwist} of $H$, denoted $H^{\Omega}$ \cite[Definition 2.5]{MTwist}. As an algebra, $H^{\Omega} = H$ but $H^{\Omega}$ has comultiplication $\Delta^{\Omega}(h) = \Omega\inv (\Delta h) \Omega$. There is also an antipode which makes $H^{\Omega}$ a Hopf algebra.

If $A$ is a left $H$-module algebra, define the twisted algebra $A_{\Omega}$ to be $A$ as a vector space but with multiplication:
\[ a \cdot_{\Omega} b = \sum (\Omega^{(1)} \cdot a) (\Omega^{(2)} \cdot b)
\]
for all $a, b \in A$. Then $A_{\Omega}$ is a left $H^{\Omega}$-module algebra using the same action of $H$ on the vector space $A$ \cite[Definition 2.6]{MTwist}. The statements in the following lemma  were used in \cite{CKWZ1} though not all stated explicitly, and so a brief argument is given.

\begin{lemma}\label{lem.cocycle} Let $H$ be a Hopf algebra, $\Omega = \sum \Omega^1 \otimes \Omega^2 \in H \otimes H$ be a normal dual cocycle, and $A$ be a left $H$-module algebra. Then the following statements hold.
\begin{enumerate}
\item $(A_{\Omega})_{\Omega \inv} = A$.
\item $(H^{\Omega})^{\Omega\inv} \cong H$.
\item $A$ is AS regular if and only if $A_{\Omega}$ is AS regular.
\item $\gldim A = \gldim A_{\Omega}$ and $\GKdim A =  \GKdim A_{\Omega}$.
\item $(A_{\Omega})^{H^{\Omega}} = A^H$.
\item If $A$ is an inner-faithful $H$-module, then $A_{\Omega}$ is an inner-faithful $H^{\Omega}$-module.
\end{enumerate}
\end{lemma}
\begin{proof}
Part (1) holds because the product $*$ in $(A_{\Omega})_{\Omega \inv}$ is given by
\[ a * b = \sum ((\Omega \inv)^{(1)} . a ) \cdot_{\Omega} ((\Omega \inv)^{(2)} . b) = \sum (\Omega^{(1)}(\Omega \inv)^{(1)} . a)(\Omega^{(2)}(\Omega \inv)^{(2)} . b) = ab.
\]

Parts (2), (3), and (4) follow from (1) and \cite[Lemma 4.12]{CKWZ1}.

Since the action of $H^{\Omega}$ on $A_{\Omega}$ is given by the same action of $H$ on the vector space $A$, we can identify the elements of $(A_{\Omega})^{H^{\Omega}}$ with those of $A^H$, we need only check that the two invariant rings have the same multiplication. But for any $a \in A$ and $b \in A^H$, we have
\[ a \cdot_{\Omega} b = \sum (\Omega^{(1)} . a) (\Omega^{(2)} . b) = \sum (\Omega^{(1)} . a) (\epsilon(\Omega^{(2)}) . b) =  \sum (\Omega^{(1)} \epsilon(\Omega^{(2)}) . a)b = ab
\]
since $\Omega$ is a normal dual cocyle. In particular, the above holds for any $a,b \in A^H$. Hence, (5) holds.

Part (6) is immediate because as algebras $H = H^{\Omega}$ and the action of $H^{\Omega}$ on $(A_{\Omega})_1$ is the same as the action of $H$ on the vector space $A_1$; furthermore, the Hopf ideals of $H$ are the same as those in $H^\Omega$. 
\end{proof}

Hence, a cotwist $H^{\Omega}$ acts on an AS regular algebra $A$ if and only if $H$ acts on the AS regular algebra $A_{\Omega\inv}$. The invariant ring $A^{H^\Omega}$ is AS regular if and only if $(A_{\Omega\inv})^H$ is AS regular.

\section{\texorpdfstring{$\bG(H) = C_2 \times C_2 \times C_2$}{G(H) = C2 x C2 x C2}} \label{c2xc2xc2}

Throughout this section, let $H$ be a semisimple Hopf algebra of dimension sixteen with
$\bG(H) = C_2 \times C_2\times C_2=\langle x\rangle\times\langle y\rangle\times\langle z\rangle$.
By Kashina's classification, there are four such algebras (see \cite[pp. 633--634]{kashina}).  We will call these algebras \kashina{1} through \kashina{4}.
They are all generated by $x,y,z$ and $w$ subject to relations:
$$wx = yw \quad wy = xw \quad wz = zw,$$
and an additional relation on $w^2$ of the form
$w^2 = f(x,y,z)$. In each case
$\epsilon(w) = 1$.  The general formula for the coproduct of $w$ is given in \cite[pp. 629--633]{kashina}.
To treat the four Hopf algebras in this case efficiently, we will consider first \twokashinas{2}{4}, and then \twokashinas{1}{3}.
Even though we treat those Hopf algebras with the same coproduct in separate subsections, there are
some similarities in the results.

\subsection{\twokashinas{2}{4}}

For the algebras  \twokashinas{2}{4} the coproduct of $w$ is given by
\begin{equation} \label{eq:K2K4coprod}
\Delta(w) = \frac{1}{2}(1 \otimes 1 + 1 \otimes xy + z \otimes 1 - z \otimes xy)(w \otimes w).
\end{equation}
For \kashina{2} (denoted $H_{d:1,1}$ in \cite{kashina}), the relation on $w^2$ is $w^2 = 1$, and for
\kashina{4} (denoted $H_{d: 1,-1} \cong (H_{b:1})^*$ in \cite{kashina}), one has $w^2 = z$.

To aid us in computing decompositions of tensor products, as well as actions on $H$-module algebras,
we have the following lemma which can be verified directly.

\begin{lemma} \label{lem:K2K4trick}
Let $H$ be \kashina{2} or \#4.
Let $A$ and $B$ be $H$-representations, and let $f \in A$, $g \in B$.  Then one has
\[
w.(f\otimes g)=\begin{cases}
w.f\otimes w.g&\text{if } zw.f=w.f \\
w.f\otimes xyw.g&\text{if }~zw.f=-w.f.
\end{cases}
\]
\end{lemma}

\begin{notation} \label{not:K1toK4reps}
Each of these algebras has two two-dimensional irreducible representations and eight one-dimensional
representations.  The two-dimensional irreducible representations for \kashina{2} are $\pi_1(u,v)$ and $\pi_2(u,v)$ defined by
\[
\pi_1(x)=\begin{bmatrix}1&0\\0&-1\end{bmatrix}\quad\pi_1(y)=\begin{bmatrix}-1&0\\0&1\end{bmatrix}\quad\pi_1(z)=\begin{bmatrix}1&0\\0&1\end{bmatrix}\quad\pi_1(w)=\begin{bmatrix}0&1\\1&0\end{bmatrix}
\]
and 
\[
\pi_2(x)=\begin{bmatrix}1&0\\0&-1\end{bmatrix}\quad\pi_2(y)=\begin{bmatrix}-1&0\\0&1\end{bmatrix}\quad\pi_2(z)=\begin{bmatrix}-1&0\\0&-1\end{bmatrix}\quad\pi_2(w)=\begin{bmatrix}0&1\\1&0\end{bmatrix}.
\]
The one-dimensional representations of \kashina{2} are $T_{(-1)^a,(-1)^a,(-1)^b,(-1)^c}(t)$ for $a,b,c \in \{0,1\}$ where the subscripts indicates the action of $x$, $y$, $z$, and $w$, respectively, on the basis element $t$. Throughout the paper, we will use similar notation for one-dimensional representations.

For \kashina{4}, the two-dimensional irreducible representations are $\pi_1(u,v)$ and $\pi_2(u,v)$ defined by
\[
\pi_1(x)=\begin{bmatrix}1&0\\0&-1\end{bmatrix}\quad\pi_1(y)=\begin{bmatrix}-1&0\\0&1\end{bmatrix}\quad\pi_1(z)=\begin{bmatrix}1&0\\0&1\end{bmatrix}\quad\pi_1(w)=\begin{bmatrix}0&1\\1&0\end{bmatrix}
\]
and 
\[
\pi_2(x)=\begin{bmatrix}1&0\\0&-1\end{bmatrix}\quad\pi_2(y)=\begin{bmatrix}-1&0\\0&1\end{bmatrix}\quad\pi_2(z)=\begin{bmatrix}-1&0\\0&-1\end{bmatrix}\quad\pi_2(w)=\begin{bmatrix}0&\i\\\i&0\end{bmatrix}.
\]
while the one-dimensional representations are $T_{(-1)^a,(-1)^a,(-1)^b,(-1)^c \i^b}(t)$ for
{$a,b,c \in \{0,1\}$} where, again, the subscripts indicate the actions of $x$, $y$, $z$, and $w$, respectively, on the basis element $t$.
\end{notation}

\begin{remark} \label{rem:K1K3sameK2K4}
The two-dimensional irreducible representations of \twokashinas{1}{3} are the same
as those of \twokashinas{2}{4}, respectively. The one-dimensional representations of \twokashinas{1}{3} and \twokashinas{2}{4} differ only in the action of $w$. The following theorem will apply to \twokashinas{1}{3},
as well as \twokashinas{2}{4}, since only the actions of $x$, $y$ and $z$ are used in the calculation.
\end{remark}

We now record the following lemma, which identifies the inner-faithful representation of smallest
dimension for the Hopf algebras in this section.

\begin{lemma} \label{C2xC2xC2inner}
  Let $H$ be a semisimple Hopf algebra of dimension sixteen with
  $\bG(H)=C_2\times C_2\times C_2$.  Let $\pi$ be an irreducible
  two-dimensional representation of $H$, and let
  $\varphi = T_{(-1)^a,(-1)^a,(-1)^b,(-1)^c\i^b}$ be a one-dimensional
  representation of $H$.  Then the action of $H$ on $\pi$ is not
  inner-faithful and the action of $H$ on $\pi \oplus \varphi$ is
  inner-faithful if and only if $b = 1$.
\end{lemma}
\begin{proof}
These assertions follow from Theorem~\ref{dim2and3InnerFaithful}, and the calculations
appearing in \cite[5.1 and 5.3]{kashina}.
\end{proof}

By the above lemma, in order for $H$ to act inner-faithfully on a three-generated algebra $A$, the action of $H$ on the degree one component $A_1$ must be of the form $\pi_j \oplus T_{(-1)^a,(-1)^a,-1,\gamma}$ where $\pi_j$
is one of the irreducible two-dimensional representations given above, and $\gamma = (-1)^c$
for Kashina \#2 and $\gamma = (-1)^c\i$ for Kashina \#4.  Furthermore, the defining
relations must be representations of the Hopf algebra in question.  Since we are considering
only quadratic AS regular algebras, we now decompose $A_1 \otimes A_1$ as a sum of irreducible $H$-representations.

We first treat the case of \kashina{2}.  A computation shows that for $i=1,2$ one has
the decomposition
\begin{equation} \label{eq:K2pidecomp}
\begin{aligned}[b]
\pi_i(u,v)^2 =&~T_{1,1,1,1}(u^2+(-1)^{i+1}v^2)\oplus T_{1,1,1,-1}(u^2-(-1)^{i+1}v^2) \oplus \\
              &~T_{-1,-1,1,1}(uv+(-1)^{i+1}vu)\oplus T_{-1,-1,1,-1}(uv-(-1)^{i+1}vu).
\end{aligned}
\end{equation}
The action of \kashina{2} by the representation $$\pi_i(u,v)\otimes T_{(-1)^a,(-1)^a,-1,(-1)^c}(t) \oplus
T_{(-1)^a,(-1)^a,-1,(-1)^c}(t)\otimes\pi_i(u,v)$$ 
on basis elements is given in the following table:
\begin{equation}\label{eq:K2actionTable}
\begin{tabular}{|c|c|c|c|c|}
\hline
& $ut$ & $vt$ & $tu$ & $tv$ \\
\hline $x$ & $(-1)^aut$& $(-1)^{a+1}vt$ & $(-1)^atu$& $(-1)^{a+1}tv$  \\ \hline
$y$ & $(-1)^{a+1}ut$ & $(-1)^avt$ & $(-1)^{a+1}tu$ & $(-1)^atv$\\ \hline
$z$ & $(-1)^iut$ & $(-1)^ivt$ & $(-1)^itu$ & $(-1)^itv$\\ \hline
$w$ & $(-1)^cvt$ & $(-1)^cut$ & $(-1)^{c+1}tv$ & $(-1)^{c+1}tu$\\ \hline
\end{tabular}
\end{equation}

Therefore if $\{i,j\} = \{1,2\}$ with $i \neq j$ then
\begin{eqnarray*}
\pi_i(u,v)\otimes T_{(-1)^a,(-1)^a,-1,(-1)^c}(t) & = &
\begin{cases} \pi_j(ut,(-1)^cvt)&a=0\\\pi_j(vt,(-1)^cut)&a=1\end{cases} \\
T_{(-1)^a,(-1)^a,-1,(-1)^c}(t)\otimes \pi_i(u,v) & = & \begin{cases} \pi_j(tu,(-1)^{c+1}tv)&a=0\\\pi_j(tv,(-1)^{c+1}tu)&a=1.\end{cases}
\end{eqnarray*}
It follows that \kashina{2} acts inner-faithfully on the algebras $A = \frac{\kk \langle u, v \rangle}{(r)}[t;\sigma]$,
where $\sigma = \begin{bmatrix} \alpha & 0 \\ 0 & -\alpha \end{bmatrix}$ for a nonzero scalar $\alpha$,
and $r = u^2 \pm v^2$ or $uv \pm vu$, where $A_1$ is the representation $\pi_i(u,v) \oplus T_{(-1)^a,(-1)^a,-1,(-1)^c}(t)$.

The case of \kashina{4} is similar.  Indeed, the decomposition of $\pi_i^2$ given in \eqref{eq:K2pidecomp} remains
the same, except the $(-1)^{i+1}$ factors are removed from the second summand of each module generator.  Further,
one may obtain the table giving the action of Kashina \#4 on
$$\pi_i(u,v)\otimes T_{(-1)^a,(-1)^a,-1,(-1)^c\i}(t) \oplus T_{(-1)^a,(-1)^a,-1,(-1)^c\i}(t)\otimes\pi_i(u,v)$$
from \eqref{eq:K2actionTable} by scaling the $w$ row by $\i^i$.
Using this modified table, one obtains the following decompositions, where $\{i,j\} = \{1,2\}$:
\begin{eqnarray*}
\pi_i(u,v)\otimes T_{(-1)^a,(-1)^a,-1,(-1)^c\i}(t) & = & \begin{cases} \pi_j(ut,(-1)^c\i^i vt)&a=0\\\pi_j(vt,(-1)^{c}\i^i ut)&a=1\end{cases} \\
T_{(-1)^a,(-1)^a,-1,(-1)^c\i}(t)\otimes \pi_i(u,v) & = & \begin{cases} \pi_j(tu,(-1)^{c+1}\i^i tv)&a=0\\\pi_j(tv,(-1)^{c+1}\i^i tu)&a=1.\\\end{cases}
\end{eqnarray*}
It follows that \kashina{4} acts on the same three-generated algebras as \kashina{2}, and we summarize these results in
the following theorem.
\begin{theorem}
\label{innerfaithful24}
\kashina{2} and \kashina{4} act inner-faithfully on the algebras
$A = \frac{\kk \langle u, v \rangle}{(r)}[t;\sigma]$,
where $\sigma = \begin{bmatrix} \alpha & 0 \\ 0 & -\alpha \end{bmatrix}$ for a nonzero scalar $\alpha$,
and $r = u^2 \pm v^2$ or $uv \pm vu$, where $A_1$ is the representation $\pi_i(u,v) \oplus T(t)$
with $i \in \{1,2\}$ and $z.t=-t$.
\end{theorem}

\subsection{\twokashinas{1}{3}} 

For the algebras \twokashinas{1}{3}, the coproduct of $w$ is given by  
\begin{equation*} \label{eq:K1K3coprod}
\Delta(w) = \frac{1}{4} \left( \sum_{b,c,\alpha,\beta \in \{0,1\}} (-1)^{b\alpha + b \beta + c \beta} y^bz^c \otimes x^{\alpha} y^{\beta}\right)(w \otimes w).
\end{equation*}
For \kashina{1} (denoted $H_{d:-1,1}$ in \cite{kashina}), one has $w^2 = \frac{1}{2}(1 + x + y - xy)$, and for
\kashina{3} (denoted $H_{d:-1,-1} \cong (H_{c:\sigma_1})^*$ in \cite{kashina}), one has $w^2 = \frac{1}{2}(1 + x + y - xy)z$.

As in the previous case, we provide a lemma which aids us in the representation theory and invariant ring calculations to follow.
\begin{lemma} \label{lem:K1K3trick}
Let $H$ be \kashina{1} or  \#3.
Let $A$ and $B$ be $H$-representations, and let $f \in A$, $g \in B$.  Then one has
\[
w.(f\otimes g)=\begin{cases}
w.f\otimes w.g&\text{if } zw.f=w.f \text{ and } yw.f = w.f,\\
w.f\otimes xw.g&\text{if }zw.f=w.f \text{ and } yw.f = -w.f,\\
w.f\otimes xyw.g&\text{if } zw.f=-w.f \text{ and } yw.f = w.f,\\
w.f\otimes yw.g&\text{if } zw.f=-w.f \text{ and } yw.f = -w.f.\\
\end{cases}
\]
\end{lemma}

As mentioned in Remark~\ref{rem:K1K3sameK2K4}, the two-dimensional representations of
Kashinas \#1 and \#3 are the same as \#2 and \#4 respectively.  The one-dimensional
representations of Kashinas \#1 and \#3 are given by $T_{(-1)^a,(-1)^a,(-1)^b,\gamma}$,
where for Kashina \#1, $\gamma = (-1)^c\i^a$ and for Kashina \#3, $\gamma = (-1)^c\i^{a+b}$.
 Lemma~\ref{C2xC2xC2inner} applies to this case as well, hence if $H$ acts inner-faithfully preserving the grading on a three-generated algebra $A$, the degree one component $A_1 = \pi_i \oplus T_{(-1)^a,(-1)^a,-1,\gamma}$
for the appropriate choice of $\gamma$.

As before, we handle the decomposition calculations for each algebra separately.
For \kashina{1}, a computation shows that the representation $\pi_i^2$ decomposes as
\begin{align*}
\pi_i(u,v)^2&=T_{1,1,1,1}(u^2+(-1)^{i+1}v^2)\oplus T_{1,1,1,-1}(u^2-(-1)^{i+1}v^2)\\
&\oplus T_{-1,-1,1,-\i}(uv+(-1)^{i+1}\i vu)\oplus T_{-1,-1,1,\i}(uv-(-1)^{i+1}\i vu).
\end{align*}
The action by the representation
$$\pi_i(u,v) \otimes T_{(-1)^a,(-1)^a,-1,(-1)^c\i^a}(t) \oplus
T_{(-1)^a,(-1)^a,-1,(-1)^c\i^a}(t) \otimes \pi_j(u,v)$$
on basis elements is given by the following table:
\begin{equation} \label{lab:K1actionTable}
\begin{tabular}{|c|c|c|c|c|}
\hline
& $ut$ & $vt$ & $tu$ & $tv$ \\
\hline $x$ & $(-1)^aut$& $(-1)^{a+1}vt$ & $(-1)^a tu$ & $(-1)^{a+1} tv$ \\ \hline
$y$ & $(-1)^{a+1}ut$ & $(-1)^avt$ & $(-1)^{a+1} tu$ & $(-1)^a tv$ \\ \hline
$z$ & $(-1)^{i}ut$ & $(-1)^{i}vt$ & $(-1)^i tu$ & $(-1)^i tv$ \\ \hline
$w$ & $(-1)^c \i^a vt$ & $(-1)^{a+c} \i^a ut$ & $(-1)^{a+c+1} \i^a tv$ & $(-1)^{c+1} \i^a tu$ \\ \hline
\end{tabular}
\end{equation}
It follows that
\[
\pi_i(u,v)\otimes T_{(-1)^a,(-1)^a,-1,(-1)^c \i^a}(t)=
\begin{cases}
\pi_j(ut, (-1)^cvt)&a=0\\
\pi_j(vt,(-1)^{c+1}\i ut)&a=1
\end{cases}
\]
\[
 T_{(-1)^a,(-1)^a,-1,(-1)^c \i^a}(t) \otimes \pi_i(u,v) =
 \begin{cases} 
 \pi_j(tu,(-1)^{c+1}tv)&a=0\\
 \pi_j(tv, (-1)^{c+1}\i tu)&a=1
 \end{cases}
\]
where $\{i,j\} = \{1,2\}$.  Therefore \kashina{1} acts inner-faithfully
on the AS regular algebras $A = \frac{\kk\langle u, v \rangle}{(r)} [t;\sigma]$ where $r = uv \pm \i vu$
or $r = u^2 \pm v^2$, $\sigma = \begin{bmatrix} \alpha & 0 \\ 0 & (-1)^{a+1}\alpha\end{bmatrix}$ for some $\alpha \in \kk^\times$,
and $A_1$ is the representation $\pi_i(u,v) \oplus T_{(-1)^a,(-1)^a,-1,(-1)^c \i^a}(t)$ with
$i \in \{1,2\}$ and $a,c \in \{0,1\}$.

The case of \kashina{3} is similar.  Indeed, the decomposition of $\pi_i^2$ given in \eqref{eq:K2pidecomp} remains
the same, except the $(-1)^{i+1}$ factors are removed from the second summand of each module generator.  Further,
one may obtain the table giving the action of \kashina{3} by
$$\pi_i(u,v)\otimes T_{(-1)^a,(-1)^a,-1,(-1)^{c}\i^{a+1}}(t) \oplus 
T_{(-1)^a,(-1)^a,-1,(-1)^c\i^{a+1}}(t)\otimes\pi_i(u,v)$$
from \eqref{eq:K2actionTable} by scaling the $w$ row by $\i^i$.
Using this modified table, one obtains the following decompositions, where $\{i,j\} = \{1,2\}$:
\[
\pi_i(u,v)\otimes T_{(-1)^a,(-1)^a,-1,(-1)^c \i^{a+1}}(t)=
\begin{cases} 
\pi_j(ut, (-1)^{i+c+1}vt)&a=0\\
\pi_j(vt,(-1)^{i+c}\i ut)&a=1\\
\end{cases}
\]
\[
 T_{(-1)^a,(-1)^a,-1,(-1)^c \i^{a+1}}(t) \otimes \pi_i(u,v) =
 \begin{cases}
 \pi_j(tu,(-1)^{i+c}tv)&a=0\\
 \pi_j(tv, (-1)^{i+c}\i tu)&a=1\\
 \end{cases}.
\]
It follows that \kashina{1} acts on the same three-generated algebras as \kashina{3}, and we summarize these results in the following theorem.
\begin{theorem}
\label{innerfaithful13}
 \kashina{1} and \kashina{3} act inner-faithfully
on the AS regular algebras $A = \frac{\kk\langle u, v \rangle}{(r)} [t;\sigma]$ where $r = uv \pm \i vu$
or $r = u^2 \pm v^2$, $\sigma = \begin{bmatrix} \alpha & 0 \\ 0 & (-1)^{a+1}\alpha\end{bmatrix}$ for some $\alpha \in \kk^\times$,
and $A_1$ is the representation $\pi_i(u,v) \oplus T_{(-1)^a,(-1)^a,-1,\gamma}(t)$ with
$i \in \{1,2\},$ $a \in \{0,1\}$, and $\gamma=(-1)^c\i^a$ for \kashina{1} or $(-1)^c\i^{a+b}$ for \kashina{3}, with $c\in\{0,1\}$.
\end{theorem}
\subsection{Invariant Rings}
In this subsection, we determine the fixed rings for each of the actions described in the previous
two subsections.

\begin{lemma} \label{lem:restrictToR}
Let $H$ be a semisimple Hopf algebra of dimension sixteen with $\bG(H) = C_2\times C_2\times C_2$,
and let $H$ act inner-faithfully on $A = \frac{\kk\langle u,v \rangle}{(r)}[t;\sigma]$ as in 
Theorems~\ref{innerfaithful24} or \ref{innerfaithful13}.  Then:
\begin{enumerate}
\item Suppose that $mt^\ell$ is an element of $A$ where $m$ is a monomial in $u$ and $v$
with respect to some fixed choice of basis of $A$, and $\ell \geq 0$ is an integer.
If $x.mt^\ell = y.mt^\ell = z.mt^\ell = mt^\ell$, then $\ell$ is even, and $m$ has even degree in $u$ and in $v$.
\item 
Let $R$ denote the subalgebra of $A$ generated by monomials of even $u$-, $v$-, and $t$-degree.
Then the subalgebra of $R$ of elements of $t$-degree zero is commutative.
\item The action of $w$ on $A$ restricts to an action on $R$, and $w$ acts as an automorphism
of $R$.
\end{enumerate}
\end{lemma}

\begin{proof}
Recall that $x.t = (-1)^at = y.t$.  If $x.mt^\ell = mt^\ell$, then one has that the $v$-degree
of $m$ is congruent to $a\ell$ modulo 2.  Similarly, the action of $y$ implies that the $u$-degree of
$m$ is congruent to $a\ell$ modulo 2.  It follows that the total degree of $m$ is even, so that $z.m = m$.
It therefore follows that $\ell$ must be even as well, which implies that $m$ has both even degree in both $u$ and $v$, proving the first claim.

By the descriptions given in the previous two subsections, the possible relations $r$ defining $A$
are $uv - \i^\beta vu$ for $\beta = 0,1,2,3$, and $u^2 - (-1)^nv^2$ for $n = 0,1$,
depending on the Hopf algebra and representation $\pi$ chosen.  The subalgebra of $A$ of consisting of elements of even $u$-degree, even $v$-degree, and $t$-degree zero 
is either $\kk[u^2,v^2]$ in the first case, or $\kk[u^2,(uv)^2,(vu)^2]$
in the second case.  A brief calculation shows that each of these rings is commutative.
Note that in the first case, we obtain a ring isomorphic to a commutative polynomial
ring, and in the second case the ring is isomorphic to $\kk[X,Y,Z]/(X^4 - YZ)$.

It is clear that $w$ acts on the subalgebra $R$.  Further, since $x$, $y$, and $z$ all act trivially
on $R$, we have that $w$ acts on $R$ as an automorphism by Lemmas~\ref{lem:K2K4trick} and
\ref{lem:K1K3trick}.
\end{proof}

We record the following lemma, which follows from direct calculation.
\begin{lemma} \label{lem:C23evenActions}
Let $H$ be a semisimple Hopf algebra of dimension sixteen with $\bG(H) = C_2\times C_2\times C_2$,
and let $H$ act inner-faithfully on $A = \frac{\kk\langle u,v \rangle}{(r)}[t;\sigma]$ as in 
Theorems~\ref{innerfaithful24} or \ref{innerfaithful13}.  Then the action of $w$
satisfies the following for $i \in \{ 1,2\}$:
\begin{center}
\begin{tabular} {|c|c|c|c|c|c|} \hline
Kashina   & $u^2$ & $v^2$ & $(uv)^2$ & $(vu)^2$ & $t^2$ \\ \hline \hline
 1 & $(-1)^{i+1}v^2$ & $(-1)^{i+1}u^2$ & $-(vu)^2$ & $-(uv)^2$ & $t^2$ \\ \hline
 2 & $(-1)^{i+1}v^2$ & $(-1)^{i+1}u^2$ & $(vu)^2$ & $(uv)^2$ & $t^2$ \\ \hline
 3 & $v^2$ & $u^2$ & $(vu)^2$ & $(uv)^2$ & $-t^2$ \\ \hline
 4 & $v^2$ & $u^2$ & $-(vu)^2$ & $-(uv)^2$ & $-t^2$ \\ \hline
\end{tabular}
\end{center}
\end{lemma}

Since $w$ acts on $R$ as an automorphism and $R$ is generated by the above elements (depending
on the relation $r$ used), we can use this table to determine the invariants.  When
$u$ and $v$ commute up to a scalar (i.e.,  $r = uv \pm vu$ or $r = uv \pm \i vu$), the following result follows immediately from the table.

\begin{theorem}
Let $H$ be a semisimple Hopf algebra of dimension sixteen with $\bG(H) = C_2\times C_2\times C_2$,
and let $H$ act inner-faithfully on $A = \frac{\kk\langle u,v \rangle}{(r)}[t;\sigma]$
as in Theorems~\ref{innerfaithful24} or \ref{innerfaithful13} where $r = uv - \i^\beta vu$.
Then every invariant is of the form:
\begin{equation*}
\begin{tabular}{|c|c|} \hline
\twokashinas{1}{2} & $\displaystyle{\sum_{j,k,m} \alpha_{j,k,m} u^{2j}v^{2j}(u^{2k} + (-1)^{ik+k}v^{2k})t^{2m}}$ \\ \hline
\twokashinas{3}{4} & $\displaystyle{\sum_{j,k,m} \alpha_{j,k,m} u^{2j}v^{2j}(u^{2k} + (-1)^{m}v^{2k})t^{2m}}$ \\ \hline
\end{tabular}
\end{equation*}
for some $\alpha_{j,k,m} \in \kk$.
\end{theorem}

When the relation is $r = u^2 - (-1)^nv^2$, we have to work a bit harder.
As mentioned in the proof of Lemma~\ref{lem:restrictToR}, the ring $R$
on which $w$ acts is the subalgebra of $A$ generated by $u^2$, $(uv)^2$
and $(vu)^2$.  We present $R$ as  $\kk[X,Y,Z][t^2;\sigma^2]/(X^4 - YZ)$
where $X = u^2$, $Y = (uv)^2$ and $Z = (vu)^2$.  A basis of $R$ is given by
$$\{X^jY^kt^{2m}~|~j,k,m \geq 0\} \cup \{X^jZ^\ell t^{2m}~|~j,m \geq 0, \ell > 0\}.$$
By Lemma~\ref{lem:C23evenActions}, one has that $w.X = (-1)^{i+n+1}X$ for \twokashinas{1}{2},
and $w.X = (-1)^nX$ for \twokashinas{3}{4}.  Since $w$ acts on $R$ as an
automorphism, we therefore have the following theorem:

\begin{theorem}
\label{invariantform1234}

Let $H$ be a semisimple Hopf algebra with $\bG(H) = C_2\times C_2\times C_2$,
and let $H$ act inner-faithfully on $A = \frac{\kk\langle u,v \rangle}{(r)}[t;\sigma]$
as in Theorems~\ref{innerfaithful24} or \ref{innerfaithful13} where $r = u^2 - (-1)^nv^2$.
Then every invariant is of the form:
\begin{equation*}
\begin{tabular}{|c|c|} \hline
\kashina{1} & $\displaystyle{\sum_{j,k,m} \alpha_{j,k,m} u^{2j}((uv)^{2k} + (-1)^{(i+n+1)j + k}(vu)^{2k})t^{2m}}$ \\ \hline
\kashina{2} & $\displaystyle{\sum_{j,k,m} \alpha_{j,k,m} u^{2j}((uv)^{2k} + (-1)^{(i+n+1)j}(vu)^{2k})t^{2m}}$ \\ \hline
\kashina{3} & $\displaystyle{\sum_{j,k,m} \alpha_{j,k,m} u^{2j}((uv)^{2k} + (-1)^{nj + m}(vu)^{2k})t^{2m}}$ \\ \hline
\kashina{4} & $\displaystyle{\sum_{j,k,m} \alpha_{j,k,m} u^{2j}((uv)^{2k} + (-1)^{nj+m+k}(vu)^{2k})t^{2m}}$ \\ \hline
\end{tabular}
\end{equation*}
for some $\alpha_{j,k,m} \in \kk$.
\end{theorem}

\begin{theorem} \label{fixedrings1234}
{Let $H$ be a semisimple Hopf algebra of dimension sixteen with $\bG(H) = C_2\times C_2\times C_2$,
and let $H$ act inner-faithfully on $A = \frac{\kk\langle u,v \rangle}{(r)}[t;\sigma]$
as in Theorems~\ref{innerfaithful24} or \ref{innerfaithful13}}, with $\pi_i$ a two-dimensional representation of $H$. 
Then the invariant subrings are:
\begin{center}
\begin{tabular}{|c|c|l|} \hline
Kashina  & Relation $r$ & Invariant ring \\ \hline
1 $(H_{d:-1,1})$ & $uv \pm \i vu$ & $\kk[u^2v^2,u^2+(-1)^{i+1}v^2,t^2]$ \\ \hline
1  $(H_{d:-1,1})$& $u^2 - (-1)^iv^2$ & $\kk[u^4,(uv)^2-(vu)^2,u^2((uv)^2+(vu)^2),t^2]$ \\ \hline
1 $(H_{d:-1,1})$ & $u^2 + (-1)^iv^2$ & $\kk[u^2,(uv)^2-(vu)^2,t^2]$ \\ \hline
2 $(H_{d:1,1})$ & $uv \pm vu$       & $\kk[u^2v^2,u^2+(-1)^{i+1}v^2,t^2]$ \\ \hline
2 $(H_{d:1,1})$ & $u^2 - (-1)^iv^2$ & $\kk[u^4,(uv)^2+(vu)^2,u^2((uv)^2-(vu)^2),t^2]$ \\ \hline
2 $(H_{d: 1,1})$ & $u^2 + (-1)^iv^2$ & $\kk[u^2,(uv)^2+(vu)^2,t^2]$ \\ \hline
3 $(H_{d:-1,-1})$ & $uv \pm \i vu$    & $\kk[u^2v^2,u^2+v^2,(u^2-v^2)t^2,t^4]$ \\ \hline
3 $(H_{d:-1,-1})$ & $u^2 - v^2$ & $\kk[u^2,(uv)^2+(vu)^2,((uv)^2-(vu)^2)t^2,t^4]$ \\ \hline
3 $(H_{d:-1,-1})$ & $u^2 + v^2$ & $\kk\!\!\left[\begin{array}{ll}u^4,(uv)^2+(vu)^2,((uv)^2-(vu)^2)t^2,\\u^2((uv)^2-(vu)^2),u^2t^2,t^4\end{array}\right]$ \\ \hline
4  $(H_{d:1,-1})$  & $uv \pm vu$       & $\kk[u^2v^2,u^2+v^2,(u^2-v^2)t^2,t^4]$ \\ \hline
4  $(H_{d:1,-1})$ & $u^2 - v^2$ & $\kk[u^2,(uv)^2-(vu)^2,((uv)^2+(vu)^2)t^2,t^4]$ \\ \hline
4 $(H_{d:1,-1})$ & $u^2 + v^2$ & $\kk\!\!\left[\begin{array}{ll}u^4,(uv)^2-(vu)^2,u^2((uv)^2+(vu)^2),\\((uv)^2+(vu)^2)t^2,u^2t^2,t^4\end{array}\right]$ \\ \hline
\end{tabular}
\end{center}
~\\
Hence the rings of invariants  in rows 1, 3, 4, and 6 are AS regular because they are Ore extensions of commutative polynomial rings, while the others are not AS regular by Lemma~\ref{notregular}.  Thus \twokashinas{1}{2} are reflection Hopf algebras for a three-dimensional AS regular algebra.
\end{theorem}

\begin{proof}
These results follow from the results in \cite[Subsection 1.4]{FKMW}.  We include details
of how to apply these results in the case of Kashina \#3, when the relation is  $u^2 + v^2.$
In this case (using the notation from \emph{op. cit.}), one takes $x = (uv)^2$, $y = -(uv)^2$,
$z = u^2$, $w = t^2$, $k = 2$ and $\alpha = -1$.  The generators provided by the lemma
are precisely the generators listed in this case.
\end{proof}

\section{\texorpdfstring{$\bG(H) = C_4 \times C_2$}{G(H) = C4 x C2}} \label{c4xc2}
Throughout this section, let $H$ be a semisimple Hopf algebra of dimension sixteen with $\bG(H) = C_4 \times C_2 = \grp{x} \times \grp{y}$.
There is a third algebra generator $z$ of $H$ with $\epsilon(z) = 1$ with coproduct
\[ \Delta(z) = \frac{1}{2}\left(1 \otimes 1 + y \otimes 1 + 1 \otimes x^2 - y \otimes x^2 \right)(z \otimes z).
\]

There are seven non-isomorphic Hopf algebras in this family each of which has different algebra relations involving the element $z$.
We list these relations in the following table, grouped according to our treatment of these algebras.  In each case we give the notation for the Hopf algebra used in \cite{kashina}, the relations in $H$, and the location in \cite{kashina} where the algebra is discussed.

\begin{equation} \label{eqn:C4C2relations}
\begin{tabular}{|c|c|c|c|c|} \hline
  Kashina & \multicolumn{3}{c|}{Relations}& Page in \cite{kashina} \\ \hline \hline
5 ($H_{c: \sigma_1}$) & $zx = xz$   & $zy = x^2yz$ & $z^2 = \frac{1+\i}{2} + \frac{1-\i}{2}x^2$ & p. 629 (2)\\ \hline
7 ($H_{c: \sigma_0}$) & $zx = xz$   & $zy = x^2yz$ & $z^2 = x(\frac{1+\i}{2} + \frac{1-\i}{2}x^2)$ & p. 629 (1) \\ \hline \hline
6 ($H_{b: 1}$)& $zx = x^3z$ & $zy = yz$ & $z^2 = 1$& p. 627 (1) \\ \hline
10 ($H_{b: y}$) & $zx = x^3z$ & $zy = yz$ & $z^2 = y$ & p. 627 (2)\\ \hline
11 ($H_{b: x^2y}$) & $zx = x^3z$ & $zy = yz$ & $z^2 = x^2y$ & p. 628 (3)\\ \hline \hline
8 ($H_{a: 1}$)& $zx = xyz$ & $zy = yz$ & $z^2 = 1$ & p. 626 (1) \\ \hline
9 ($H_{a: y}$)& $zx = xyz$ & $zy = yz$ & $z^2 = y$ & p. 626 (2)\\ \hline
\end{tabular}
\end{equation}

We now record the following lemma, which identifies the inner-faithful representation of smallest
dimension for the Hopf algebras in this section.


\begin{lemma} \label{C4xC2inn}
  Let $H$ be a semisimple Hopf algebra of dimension sixteen with
  $\bG(H)=C_4\times C_2$.  Let $\pi$ be an irreducible
  two-dimensional representation of $H$, and let
  $\varphi = T_{\alpha,\beta,\gamma}$ be a one-dimensional
  representation of $H$.  Then:
  \begin{enumerate}
    \item In \twokashinas{5}{7}, the action of $H$ on $\pi$ is inner-faithful.
    \item In \threekashinas{6}{10}{11}, the action of $H$ on $\pi$ is not inner-faithful
      and the action of $H$ on $\pi \oplus \varphi$ is inner-faithful if and only if $\beta = -1$.
    \item In \twokashinas{8}{9}, the action of $H$ on $\pi$ is not inner-faithful
      and the action of $H$ on $\pi \oplus \varphi$ is inner-faithful if and only if $\alpha = \pm\i$.       
  \end{enumerate}
\end{lemma}
\begin{proof}
These assertions follow from Theorem~\ref{dim2and3InnerFaithful}, and the calculations
appearing in \cite[Section 5]{kashina}.
\end{proof}

The next lemma is a version of Lemmas~\ref{lem:K2K4trick} and \ref{lem:K1K3trick} for
the class of Hopf algebras under consideration in this section.
\begin{lemma} \label{lem:trickC_4C_2}
Let $H$ be a semisimple Hopf algebra of dimension sixteen with $\bG(H)=C_4\times C_2$, as described at the beginning of this section.
Let $B$ and $C$ be $H$-representations, and let $f \in B$, $g \in C$.  Then one has
\[
z.(f\otimes g)=\begin{cases}
z.f\otimes z.g & \text{if } yz.f=z.f\\
z.f\otimes x^2z.g & \text{if } yz.f=-z.f.
\end{cases}
\]
\end{lemma}

\subsection{\twokashinas{5}{7}}
The one-dimensional representations for \kashina{5} are $T_{\pm1, \pm 1, \pm 1}(t)$, where the subscripts indicate the action of $x$, $y$, and $z$, respectively, on the basis element $t$. The irreducible two-dimensional representations are $\pi_1(u,v)$ and $\pi_2(u,v)$ which are defined by
\[ \pi_1(x) = \begin{bmatrix}\i & 0 \\ 0 &\i  \end{bmatrix} \quad \pi_1(y) = \begin{bmatrix} 1 & 0 \\ 0 & -1 \end{bmatrix} \quad  \pi_1(z) = \begin{bmatrix} 0 & \omega \\ \omega & 0\end{bmatrix}
\] and 
\[ \pi_2(x) = \begin{bmatrix} -\i & 0

\\  0 & -\i  \end{bmatrix} \quad \pi_2(y) = \begin{bmatrix} 1 & 0 \\ 0 & -1 \end{bmatrix} \quad  \pi_2(z) = \begin{bmatrix} 0 & \omega \\ \omega & 0 \end{bmatrix},
\]
where $\omega = e^{\frac{2 \pi \i}{8}}$.
The one-dimensional representations for \kashina{7} are $T_{1, \pm 1, \pm 1}(t)$ and $T_{-1, \pm 1, \pm \i}(t)$ while the two-dimensional irreducible representations are defined by
\[ \pi_1(x) = \begin{bmatrix}\i & 0 \\ 0 &\i  \end{bmatrix} \quad \pi_1(y) = \begin{bmatrix} 1 & 0 \\ 0 & -1 \end{bmatrix} \quad  \pi_1(z) = \begin{bmatrix} 0 & 1 \\ -1 & 0\end{bmatrix}
\]
and
\[ \pi_2(x) = \begin{bmatrix} -\i & 0 \\ 0 & -\i \end{bmatrix} \quad \pi_2(y) = \begin{bmatrix} 1 & 0 \\ 0 & -1 \end{bmatrix} \quad  \pi_2(z) = \begin{bmatrix} 0 & -1 \\ 1 & 0 \end{bmatrix}.
\]
The irreducible two-dimensional representations of \twokashinas{5}{7} can be written as
\[
\pi(x)=\begin{bmatrix} \i^a&0\\0&\i^a\end{bmatrix},\quad \pi(y)=\begin{bmatrix}1&0\\0&-1\end{bmatrix},\quad\pi(z)=\begin{bmatrix}0&\omega^b\\\omega^c&0\end{bmatrix},
\]
with $a\in\{1,3\}, b,c\in\{0,1,4\}$ and $\omega=e^{\frac{2\pi\i}{8}}$. 

One computes that when these Hopf algebras act on an algebra $A$ with $A_1 = \pi(u,v)$, then the actions on the degree two elements are given by the following table:
\begin{center}
\begin{tabular}{|c|c|c|c|c| }
\hline
& $u^2$ & $v^2$ & $uv$ & $vu$ \\
\hline $x$ & $-u^2$& $-v^2$ & $-uv$ & $-vu$ \\ \hline
$y$ & $u^2$ & $v^2$ & $-uv$ & $-vu$ \\ \hline
$z$ & $-\i^{c}v^2$ & $\i^b u^2$ & $-\omega^{b+c}{vu}$& ${\omega^{b+c}uv}$\\
\hline
\end{tabular}
\end{center}
Since either $c=b=1$ or one is 0 and the other is 4, it follows that
\begin{align*} \pi(u,v)  \otimes \pi(u,v) &= T_{-1,1,\i^{b+1}}(u ^2 + \i v^2) 
\oplus T_{-1,1,-\i^{b+1}}(u^2 - \i v^2) \\
&\hphantom{=}\oplus T_{-1,-1,\i\omega^{b+c}} (uv + \i vu) \oplus T_{-1,-1,{-\i\omega^{b+c}}} (uv- \i vu).
\end{align*}
Therefore these algebras act on quadratic AS regular algebras of the form $A=\frac{\kk\langle u,v\rangle}{(r)}$ where $A_1=\pi$ and $r=u^2-(-1)^f \i v^2$ or $r=uv-(-1)^e \i vu$, with $e,f\in\{0,1\}$. These actions are inner-faithful by
 Lemma~\ref{C4xC2inn}. 

An induction argument shows that 
\[
z.(u^n)=\omega^{cn+4a\binom{n}{2}}v^n,\quad z.(v^n)=\omega^{nb}u^n.
\]

We first focus on the relation $uv-(-1)^e\i vu$. The action of $y$ implies that any invariant must have even $v$-degree. Let $F$ be an invariant and write $F=\sum\alpha_{l,k}u^lv^{2k}$. The action of $x$ implies that $al+2ak\equiv0 \pmod{4}$, which implies that $l$ is even. Replacing $l$ with $2l$, we get an invariant of the form $\sum\alpha_{l,k}u^{2l}v^{2k}$, satisfying $l\equiv k\pmod{2}$.
Using the formulae above, it follows that
\[
z.(u^{2l}v^{2k})=\i^{kb+cl+2al}u^{2k}v^{2l},
\]
and therefore $F$ must be of the form
\[
F=\sum_{l\equiv k  \ (\mathrm{mod} \ 2)}\alpha_{l,k}(u^{2l}v^{2k}+\i^{kb+cl+2al}u^{2k}v^{2l}),
\]
which can be rewritten as
\[
F=\sum_{k,l \geq 0}\alpha_{l,k}u^{2k}v^{2k}(u^{4l}+(-1)^{cl+k}\i^{k(b+c)}v^{4l}).
\]

A similar computation shows that for the relation $u^2-(-1)^f\i v^2$, every invariant must be of the form
\[
F=\sum_{k,l \geq 0}\alpha_{l,k}u^{4l}((uv)^{2k}+(-1)^{l(c+1)}\i^{(b+c)k}(vu)^{2k}).
\]
The next theorem follows from Remark 1.14, \cite[Lemma 1.10]{FKMW}, and the formulae above.

\begin{theorem} \label{fixedrings57}
Let $H$ be \kashina{5} or \kashina{7}. Then $H$ acts inner-faithfully on the algebras $A=\frac{\kk\langle u,v\rangle}{(r)}$ with $A_1 = \pi = \kk u\oplus\kk v$ an irreducible two-dimensional representation and $r=u^2\pm\i v^2$ or $r=uv\pm\i vu$. The invariant rings for these actions are: 
\begin{center}
\begin{tabular}{|c|c|l| }
\hline
Kashina  & Relation $r$ & Invariant  ring  \\
\hline 
$5$  & $uv\pm\i vu$ & $\kk[u^2v^2,u^4-v^4]$\\ \hline
$5$  & $u^2\pm\i v^2$ & $\kk[u^4,(uv)^2-(vu)^2]$\\ \hline
$7$ & $uv\pm\i vu$ & $\kk[u^4v^4,u^4+v^4,u^2v^2(u^4+v^4)]$ \\ \hline
$7$ & $u^2\pm\i v^2$& $\kk[u^8,(uv)^2+(vu)^2,u^4((uv)^2-(vu)^2)]$ \\ \hline
\end{tabular}
\end{center}

The  invariant rings for \kashina 5 are commutative polynomial rings, hence \kashina{5} is a reflection Hopf algebra for the algebras and actions given while, by Lemma~\ref{notregular}, \kashina{7} is not.
\end{theorem}

\subsection{\threekashinas{6}{10}{11}}
The one-dimensional representations for \kashina{6} are $T_{\pm1, \pm 1, \pm 1}(t)$, where the subscripts indicate the action of $x$, $y$, and $z$, respectively, on the basis element $t$. The irreducible two-dimensional representations are $\pi_1(u,v)$ and $\pi_2(u,v)$ defined by
\[ \pi_1(x) = \begin{bmatrix}\i & 0 \\ 0 &-\i  \end{bmatrix} \quad \pi_1(y) = \begin{bmatrix} 1 & 0 \\ 0 & 1 \end{bmatrix} \quad  \pi_1(z) = \begin{bmatrix} 0 & 1 \\ 1 & 0\end{bmatrix}
\] 
and 
\[ \pi_2(x) = \begin{bmatrix}\i & 0 \\ 0 &-\i  \end{bmatrix} \quad \pi_2(y) = \begin{bmatrix} -1 & 0 \\ 0 & -1 \end{bmatrix} \quad  \pi_2(z) = \begin{bmatrix} 0 & 1 \\ 1 & 0\end{bmatrix}.
\]
The one-dimensional representations for \kashina{10} are $T_{\pm 1, 1, \pm 1}$ and $T_{\pm 1, -1, \pm \i}(t)$ and the irreducible two-dimensional representations are defined by
\[ \pi_1(x) = \begin{bmatrix}\i & 0 \\ 0 &-\i  \end{bmatrix} \quad \pi_1(y) = \begin{bmatrix} 1 & 0 \\ 0 & 1 \end{bmatrix} \quad  \pi_1(z) = \begin{bmatrix} 0 & 1 \\ 1 & 0\end{bmatrix}
\]
and
\[ \pi_2(x) = \begin{bmatrix}\i & 0 \\ 0 &-\i  \end{bmatrix} \quad \pi_2(y) = \begin{bmatrix} -1 & 0 \\ 0 & -1 \end{bmatrix} \quad  \pi_2(z) = \begin{bmatrix} 0 & \i \\ \i & 0\end{bmatrix}.
\]
The one-dimensional representations for \kashina{11} are $T_{\pm 1, 1, \pm 1}(t)$ and $T_{\pm 1, -1, \pm \i}(t)$ and the irreducible two-dimensional representations are defined by
\[ \pi_1(x) = \begin{bmatrix}\i & 0 \\ 0 &-\i  \end{bmatrix} \quad \pi_1(y) = \begin{bmatrix} 1 & 0 \\ 0 & 1 \end{bmatrix} \quad  \pi_1(z) = \begin{bmatrix} 0 & -1 \\ 1 & 0\end{bmatrix}
\]
and
\[ \pi_2(x) = \begin{bmatrix}\i & 0 \\ 0 &-\i  \end{bmatrix} \quad \pi_2(y) = \begin{bmatrix} -1 & 0 \\ 0 & -1 \end{bmatrix} \quad  \pi_2(z) = \begin{bmatrix} 0 & -\i \\ \i & 0 \end{bmatrix}.
\]

The two-dimensional representations for \threekashinas{6}{10}{11} can be written as
\[
\pi(x) = \begin{bmatrix}\i & 0 \\ 0 &-\i  \end{bmatrix} \quad \pi(y) = \begin{bmatrix} (-1)^a & 0 \\ 0 & (-1)^a \end{bmatrix} \quad  \pi(z) = \begin{bmatrix} 0 & \i^c \\ \i^b & 0\end{bmatrix}
\]
for $a\in\{0,1\}$ and $(b,c) \in \{(0,0), (0,2), (1,1), (1,3) \}$. The one-dimensional representations are $T_{(-1)^e,(-1)^g,\i^{2f+\epsilon g}}$ with $e,f,g\in\{0,1\}$  and $\epsilon = 0$ for \kashina{6}, $\epsilon =1$ for \twokashinas{10}{11}.

To compute $\left(\pi(u,v) \oplus T_{(-1)^e,(-1)^g,\i^{2f+\epsilon g}}(t) \right) \otimes \left( \pi(u,v) \oplus T_{(-1)^e,(-1)^g,\i^{2f+\epsilon g}}(t)\right)$ one checks that the actions of $x$, $y$, and $z$ are given by the following tables.
\begin{center}
\begin{tabular}{|c|c|c|c|c| }
\hline
& $u^2$ & $v^2$ & $uv$ & $vu$ \\
\hline $x$ & $-u^2$& $-v^2$ & $uv$ & $vu$ \\ \hline
$y$ & $u^2$ & $v^2$ & $uv$ & $vu$ \\ \hline
$z$ & $(-1)^{a+b}v^2$ & $(-1)^{a+c}u^2$ & $(-1)^a\i^{b+c}vu$& $(-1)^a\i^{b+c}uv$\\
\hline
\end{tabular}
\end{center}

\begin{center}
\begin{tabular}{|c|c|c|c|c|c| }
\hline
 & $ut$& $tu$ & $vt$ & $tv$ & $t^2$ \\
\hline $x$  & $(-1)^e\i ut$ & $(-1)^e\i tu$ & $(-1)^{e+1}\i vt$ & $(-1)^{e+1}\i tv$ & $t^2$\\ \hline
$y$ &  $(-1)^{a+g}ut$ & $(-1)^{a+g}tu$ & $(-1)^{a+g}vt$ & $(-1)^{a+g}tv$ & $t^2$\\ \hline
$z$ & $(-1)^f \i^{b+\epsilon g} vt$ & $(-1)^{f+g}\i^{b+\epsilon g} tv$ & $(-1)^f \i^{c+\epsilon g} ut$ & $(-1)^{f+g}\i^{c+\epsilon g} tu$ & $(-1)^{\epsilon g}t^2$\\
\hline
\end{tabular}
\end{center}

Therefore
\begin{align*}&\left(\pi(u,v) \oplus T_{(-1)^e,(-1)^f,\i^f}(t) \right) \otimes \left( \pi(u,v) \oplus T_{(-1)^e,(-1)^f,\i^f}(t)\right) \\
&=T_{-1,1,(-1)^{a+b+1}}(u ^2 - v^2)\oplus T_{-1,1,(-1)^{a+b}}(u^2 +v^2) \\&\oplus T_{1,1,(-1)^a\i^{b+c}} (uv + vu) \oplus T_{1,1,(-1)^{a+1}\i^{b+c}} (uv-vu) \oplus T_{1,1,(-1)^f}(t^2)\\
&\oplus \text{2 two-dimensional representations.}
\end{align*}

The two two-dimensional representations appearing in the decomposition above are given in the next table; we give the decomposition for a specific $T_{(-1)^e,(-1)^g,\i^{2f+\epsilon g}}$ (namely, $e=0, g=1,f =0$). In the remaining cases, there are similar decompositions.

\begin{center}
\begin{tabular}{|c|c|c|c| }
\hline
Kashina & $\pi$ & $T$ & Decomposition\\
\hline  
6& $\pi_1$ & $T_{1,-1,1}$ & $\pi_2(ut,vt)\oplus\pi_2(tu,-tv)$\\ \hline
6& $\pi_2$ & $T_{1,-1,1}$  & $\pi_1(ut,vt)\oplus\pi_1(tu,-tv)$\\ \hline
10 &$\pi_1$ & $T_{1,-1,\i}$ & $\pi_2(ut,vt)\oplus\pi_2(tu,-tv)$\\ \hline
10 & $\pi_2$ & $T_{1,-1,\i}$ &$\pi_1(ut,-vt)\oplus\pi_1(tu,tv)$\\ \hline
11 & $\pi_1$ & $T_{1,-1,\i}$ & $\pi_2(ut,vt)\oplus \pi_2(tu,-tv)$\\ \hline
11 & $\pi_2$ & $T_{1,-1,\i}$ & $\pi_1(ut,-vt)\oplus\pi_1(tu,tv)$\\ \hline
\end{tabular}
\end{center}


Therefore, \threekashinas{6}{10}{11} act inner-faithfully on quadratic regular algebras of the form $A=\frac{\kk\langle u,v\rangle}{(r)}[t;\sigma]$ with $A_1 = \pi(u,v) \oplus T(t)$, where $\pi = \pi_1$ or $\pi_2$, $T =T_{(-1)^e, -1,\i^{2f+\epsilon }}$, for $e,f \in \{0,1\}$, $\sigma=\begin{bmatrix}\alpha&0\\0&-\alpha\end{bmatrix}$ with $\alpha \in \kk^\times$, and $r=uv-(-1)^qvu$ or $r=u^2-(-1)^pv^2$, with $p,q\in\{0,1\}$.

\begin{lemma} \label{evenTdeg61011}
Let $\pi(u,v) \oplus T(t)$ be a three-dimensional inner-faithful representation
of $H$, where $H$ is one of \threekashinas{6}{10}{11} as in the statement of 
  Lemma~\ref{C4xC2inn},  and let $A$
be a graded $\kk$-algebra with $A_1 = \pi \oplus T$ such that $H$ acts on $A$.
Suppose that $T = T_{\alpha,\beta,\gamma}$.  Let $m$ be a monomial in $u$ and $v$ in
the algebra $A$.  If $x.mt^\ell = y.mt^\ell = mt^\ell$, then $\ell$ is even.
\end{lemma}

\begin{proof}
  One has that $\alpha = \pm 1$ and by Lemma~\ref{C4xC2inn},
  $\pi \oplus T$ is inner-faithful if and only if $\beta = -1$.  If
  $x.mt^\ell = mt^\ell$, then one has that $\deg_u m - \deg_v m$ is
  even.  Furthermore, if $y.mt^\ell = mt^\ell$, then one has that
  $\ell$ is even when $\pi = \pi_1$ and that $\ell + \deg m$ is even
  if $\pi = \pi_2$.  In either case, it follows that $\ell$ is even.
\end{proof}

An inductive argument shows that
\[
z.(u^n)=\i^{cn+2a\binom{n}{2}}v^n,\quad z.(v^m)=\i^{bm+2a\binom{m}{2}}u^m.
\]
We consider the relation $uv-(-1)^qvu$.  By Lemma~\ref{evenTdeg61011},
an invariant must be of the form
$F=\sum\alpha_{n,m,k}u^nv^mt^{2k}$. To be invariant under $x$ we must
have $n\equiv m \pmod{4}$; if this condition is satisfied
then $F$ is also invariant under $y$. It follows from the formulae
above that if $n\equiv m \pmod{4}$, then
\[
z.(u^nv^mt^{2k})=\begin{cases} (-1)^{fk}\i^{(c+b)n}u^mv^nt^{2k}& \text{if $n$ is even}\\
(-1)^{fk+q}\i^{(c+b)n}u^mv^nt^{2k} & \text{if $n$ is odd and $a = 0$}\\
-(-1)^{fk+q}\i^{(c+b)n}u^mv^nt^{2k} & \text{if $n$ is odd and $a = 1$.}\end{cases}
\]
When $n\equiv m \pmod{4}$, we can write this as $z.(u^nv^mt^{2k})=(-1)^{fk+(q+a)n}\i^{(c+b)n}u^mv^nt^{2k}$. Using the formula above it follows that an invariant must be of the form
\[
\sum\alpha_{n,m,k}u^mv^m(u^{4n}+(-1)^{fk+(q+a)m}\i^{(c+b)m}v^{4n})t^{2k}.
\]

Similarly if the relation is $u^2-(-1)^pv^2$, then an invariant must be of the form
\[
\sum\alpha_{n,m,k}u^{4m}((uv)^n+(-1)^{an+fk}\i^{(b+c)n}(vu)^n)t^{2k}.
\]
The next theorem follows from the results in  \cite[Subsection 1.4]{FKMW} and the formulae above.

\begin{theorem}
\label{fixedring61011}
Let $H$ be Kashina \#6, \#10, or \#11. Then $H$ acts inner-faithfully on the algebras $A=\frac{\kk\langle u,v\rangle}{(r)}[t;\sigma]$ with $A_1 = \pi(u,v) \oplus T(t)$ where $\pi$ is an irreducible two-dimensional representation given in the table below, $T$ is a one-dimensional representation with $y.t=-t$, $\sigma=\begin{bmatrix}\alpha&0\\0&-\alpha\end{bmatrix}$, for $\alpha \in \kk^\times$, and $r=uv\pm vu$ or $r=u^2\pm v^2$. The invariant rings for these actions are also given in the table below.

\begin{center}
\begin{tabular}{|c|c|c|l| }
\hline
Kashina & $\pi$ & Relation $r$ & Invariant ring\\
\hline  
6& $\pi_1$ & $uv-vu$ & $\kk[uv,u^4+v^4,t^2]$\\ \hline
6& $\pi_1$ & $uv+vu$ & $\kk[u^2v^2,u^4+v^4,uv(u^4-v^4),t^2]$\\ \hline
6& $\pi_2$ & $uv-vu$ & $\kk[u^2v^2,u^4+v^4,uv(u^4-v^4),t^2]$\\ \hline
6& $\pi_2$ & $uv+vu$ & $\kk[uv,u^4+v^4,t^2]$\\ \hline
10& $\pi_1$ or $\pi_2$ & $uv-vu$ & $\kk[uv,u^4+v^4,(u^4-v^4)t^2,t^4]$\\ \hline
10& $\pi_1$ or $\pi_2$ & $uv+vu$ & $\kk\!\!\left[\begin{array}{ll}u^2v^2,u^4+v^4,uv(u^4-v^4),\\uvt^2,(u^4-v^4)t^2,t^4\end{array}\right]$\\ \hline
11& $\pi_1$ or $\pi_2$ & $uv-vu$ &
$\kk\!\!\left[\begin{array}{ll}u^2v^2,u^4+v^4,uv(u^4-v^4),\\uvt^2,(u^4-v^4)t^2,t^4\end{array}\right]$\\ \hline
11& $\pi_1$ or $\pi_2$ & $uv+vu$ &
$\kk[uv,u^4+v^4,(u^4-v^4)t^2,t^4]$\\ \hline
6& $\pi_1$ & $u^2\pm v^2$ & $\kk[u^4,uv+vu,t^2]$\\ \hline
6& $\pi_2$ & $u^2\pm v^2$ & $\kk[u^4,uv-vu,t^2]$\\ \hline
10& $\pi_1$ or $\pi_2$ & $u^2\pm v^2$ & $\kk[u^4,uv+vu,(uv-vu)t^2,t^4]$\\ \hline
11& $\pi_1$ or $\pi_2$ & $u^2\pm v^2$ & $\kk[u^4,uv-vu,(uv+vu)t^2,t^4]$\\ \hline
\end{tabular}
\end{center}
\kashina{6} is a reflection Hopf algebra in the cases listed in the
first and fourth rows of the table above, as the invariant rings are
Ore extensions of a commutative polynomial ring; it is not a
reflection Hopf algebra in the cases given in the second and third
rows of the table above, as the invariant rings are not AS regular by
Lemma~\ref{notregular}. By Lemma~\ref{notregular} \kashina{10} and \kashina{11} are not
reflection Hopf algebras for any of the algebras and representations
described above.
\end{theorem}

\subsection{\twokashinas{8}{9}}
The one-dimensional representations for \kashina{8} are $T_{\pm 1, 1, \pm 1}(t)$ and $T_{\pm \i, 1, \pm 1}(t)$, where the subscripts indicate the actions of $x$, $y$, and $z$, respectively, on the basis element $t$. The irreducible two-dimensional representations are defined by
\[ \pi_1(x) = \begin{bmatrix}\i & 0 \\ 0 &-\i  \end{bmatrix} \quad \pi_1(y) = \begin{bmatrix} -1 & 0 \\ 0 & -1 \end{bmatrix} \quad  \pi_1(z) = \begin{bmatrix} 0 & 1 \\ 1 & 0\end{bmatrix}
\]
and
\[ \pi_2(x) = \begin{bmatrix}1 & 0 \\ 0 &-1  \end{bmatrix} \quad \pi_1(y) = \begin{bmatrix} -1 & 0 \\ 0 & -1 \end{bmatrix} \quad  \pi_1(z) = \begin{bmatrix} 0 & 1 \\ 1 & 0\end{bmatrix}.
\]
The one-dimensional representations for \kashina{9} are $T_{\pm1,1,\pm1}(t)$ and $T_{\pm \i, 1, \pm 1}(t)$ and the irreducible two-dimensional representations are given by
\[ \pi_1(x) = \begin{bmatrix}\i & 0 \\ 0 &-\i  \end{bmatrix} \quad \pi_1(y) = \begin{bmatrix} -1 & 0 \\ 0 & -1 \end{bmatrix} \quad  \pi_1(z) = \begin{bmatrix} 0 & -1 \\ 1 & 0\end{bmatrix}
\]
and
\[ \pi_2(x) = \begin{bmatrix} 1 & 0 \\ 0 & -1 \end{bmatrix} \quad \pi_1(y) = \begin{bmatrix} -1 & 0 \\ 0 & -1 \end{bmatrix} \quad  \pi_1(z) = \begin{bmatrix} 0 & -1 \\ 1 & 0 \end{bmatrix}.
\]
The irreducible two-dimensional representations for \twokashinas{8}{9} can be written as 
\[
\pi(x)=\begin{bmatrix}\i^a&0\\0&-\i^a\end{bmatrix},\quad \pi(y)=\begin{bmatrix}-1&0\\0&-1\end{bmatrix},\quad \pi(z)=\begin{bmatrix}0&(-1)^b\\1&0\end{bmatrix},
\]
with $a,b\in\{0,1\}$. The one-dimensional representations can be written as $T_{(\i)^c,1,(-1)^d}$ with $c \in \{0,1,2,3\}$ and $d \in \{0,1\}$.

To compute $\left(\pi(u,v) \oplus T_{(\i)^c,1,(-1)^d}(t) \right) \otimes \left( \pi(u,v) \oplus T_{\i^c,1,(-1)^d}(t)\right)$ one checks that
\begin{center}
\begin{tabular}{|c|c|c|c|c| }
\hline
& $u^2$ & $v^2$ & $uv$ & $vu$ \\
\hline $x$ & $(-1)^au^2$& $(-1)^av^2$ & $(-1)^{a+1}uv$ & $(-1)^{a+1}vu$ \\ \hline
$y$ & $u^2$ & $v^2$ & $uv$ & $vu$ \\ \hline
$z$ & $(-1)^{a}v^2$ & $(-1)^{a}u^2$ & $(-1)^{a+b}vu$& $(-1)^{a+b}uv$\\
\hline
\end{tabular}
\end{center}

\begin{center}
\begin{tabular}{|c|c|c|c|c|c| }
\hline
 & $ut$& $tu$ & $vt$ & $tv$ & $t^2$ \\
\hline $x$  & $\i^{a+c} ut$ & $\i^{a+c} tu$ & $ - \i^{a+c} vt$ & $- \i^{a+c} tv$ & $(-1)^ct^2$\\ \hline
$y$ &  $-ut$ & $-tu$ & $-vt$ & $-tv$ & $t^2$\\ \hline
$z$ & $(-1)^{c+d} vt$ & $(-1)^d tv$ & $(-1)^{b+c+ d} ut$ & $(-1)^{b+d} tu$ & $t^2$\\
\hline
\end{tabular}
\end{center}

Therefore we have the decomposition
\begin{align*}&\left(\pi(u,v) \oplus T_{(\i)^c,1,(-1)^d}(t) \right) \otimes \left( \pi(u,v) \oplus T_{(\i)^c,1,(-1)^d}(t)\right) \\
&=T_{(-1)^a,1,(-1)^{a+1}}(u ^2 - v^2)\oplus T_{(-1)^a,1,(-1)^{a}}(u^2 +v^2) \\&\oplus T_{(-1)^{a+1},1,(-1)^{a+b}} (uv + vu) \oplus T_{(-1)^{a+1},1,(-1)^{a+b+1}} (uv-vu) \oplus T_{-1,1,1}(t^2)\\
&\oplus \text{2 two-dimensional representations.}
\end{align*}
The two two-dimensional representations appearing in the decomposition above are given in the next table for $c=1,  d=0$; in the remaining cases there are similar decompositions.

\begin{center}
\begin{tabular}{|c|c|c|c| }
\hline
Kashina & $\pi$  & Decomposition\\
\hline  
8& $\pi_1$  & $\pi_2(vt,-ut)\oplus\pi_2(tv,tu)$\\ \hline
8& $\pi_2$   & $\pi_1(ut,-vt)\oplus\pi_1(tu,tv)$\\ \hline
9 &$\pi_1$  & $\pi_2(vt,ut)\oplus\pi_2(tv,-tu)$\\ \hline
9 & $\pi_2$ &$\pi_1(ut,-vt)\oplus\pi_1(tu,tv)$\\ \hline
\end{tabular}
\end{center}

Therefore \twokashinas{8}{9} act inner-faithfully on quadratic regular algebras of the form $A=\frac{\kk\langle u,v\rangle}{(r)}[t;\sigma]$ where $A = \pi(u,v) \oplus T(t)$ with $\pi = \pi_1$ or $\pi_2$, $T = T_{(-1)^c\i,1,(-1)^d}$ for $c,d \in \{ 0,1\}$, $\sigma=\begin{bmatrix}\alpha&0\\0&-\alpha\end{bmatrix}$ with $\alpha \in \kk^\times$, and $r=uv-(-1)^evu$ or $r=u^2-(-1)^fv^2$, with $e,f\in\{0,1\}$.

\begin{lemma} \label{evenTdeg89}
Let $\pi(u,v) \oplus T(t)$ be a three-dimensional inner-faithful representation
of $H$, where $H$ is one of \twokashinas{8}{9} as in the statement of 
  Lemma~\ref{C4xC2inn},  and let $A$
be a graded $\kk$-algebra with $A_1 = \pi \oplus T$ such that $H$ acts on $A$.
Suppose that $T = T_{\alpha,\beta,\gamma}$.  Let $m$ be a monomial in $u$ and $v$ in
the algebra $A$.  If $x.mt^\ell = y.mt^\ell = mt^\ell$, then $\ell$ is even.
\end{lemma}

\begin{proof}
  One has that $\beta = 1$ and by Lemma~\ref{C4xC2inn}, $\pi \oplus T$
  is inner-faithful if and only if $\alpha = \pm \i$.  For any two-dimensional representaiton $\pi$, $y$
  scales both $u$ and $v$ by $-1$/
 Hence, if $y.mt^\ell = mt^{\ell}$ then $\deg m$, the total degree of $m$ in $u$ and $v$, is even.  The action of $x$ on $\pi$ is given by
  $x.u = \i^au$ and $x.v = -\i^av$ for $a = 0$ or $a = 1$.  Then
  $x.mt^\ell = mt^\ell$ implies that $2\deg_v m + a\deg m \pm \ell$ is
  a multiple of 4.  Since $\deg m$ is even, it follows that $\ell$
  must be even as well.
\end{proof}

An induction argument shows that 
\[
z.(u^n)=(-1)^{a\binom{n}{2}}v^n,\quad z.(v^n)=(-1)^{a\binom{n}{2}+bn}u^n.
\]
The previous two formulae imply that if $n\equiv m\pmod{2}$ then
\[
z.(u^nv^m)=(-1)^{a\frac{n+m}{2}+bm}v^nu^m.
\]
We first assume that the relation is $uv-(-1)^evu$.  By Lemma
\ref{evenTdeg89}, an invariant $F$ must be of the form
$F=\sum\alpha_{k,l,m}u^kv^lt^{2m}$. In order for $F$ to be invariant under
$y$, we must have $k\equiv l\pmod{2}$,
while in order to be invariant under $x$, we must have
$ak+(a+2)l+2m\equiv0\pmod{4}$. To be invariant under $z$,
$F$ must be of the form
\[
\sum_{\substack{k\equiv l \ (\mathrm{mod} \ 2)\\ak+(a+2)l+2m\equiv0  \ (\mathrm{mod} \ 4)}}\alpha_{k,l,m}(u^kv^l+(-1)^{a\frac{l+k}{2}+l(b+e)}u^lv^k)t^{2m}.
\]
Replacing $2m$ with $4m$ or $4m+2$ we get that $F$ must be of the form
\begin{align*}
\sum_{\substack{k\equiv l  \ (\mathrm{mod} \ 2)\\ak+(a+2)l\equiv 0  \ (\mathrm{mod} \ 4)}}\alpha_{k,l,m}(u^kv^l+(-1)^{l(1+b+e)}u^lv^k)t^{4m}+\\
\sum_{\substack{k\equiv l  \ (\mathrm{mod} \ 2)\\ak+(a+2)l \equiv2  \ (\mathrm{mod} \ 4)}}\alpha_{k,l,m}(u^kv^l-(-1)^{l(1+b+e)}u^lv^k)t^{4m+2}.
\end{align*}
One can rewrite the previous sum as
\[
F=
\begin{cases}
\displaystyle\sum_{j\equiv m \ (\mathrm{mod} \ 2)}\alpha_{j,l,m}u^jv^j(u^{2l}+(-1)^{(b+e)j}v^{2l})t^{2m}\quad \text{ for } a=0\\
\displaystyle\sum_{l\equiv m \ (\mathrm{mod} \ 2)}\alpha_{j,l,m}u^jv^j(u^{2l}+(-1)^{k(1+b+e)+l}v^{2l})t^{2m}\quad \text{ for }a=1.
\end{cases}
\]

For the relation $u^2-(-1)^fv^2$ one can prove by induction the formulae
\[
z.((vu)^n)=(-1)^{n(a+b)}(uv)^n,\quad z.((uv)^n)=(-1)^{n(a+b)}(vu)^n,
\]
and proceeding as for the relation $uv-(-1)^evu$ it follows that an invariant must be of the form
\[
F=
\begin{cases}
\displaystyle\sum_{m\equiv k \ (\mathrm{mod} \ 2)}\alpha_{n,m,k}u^{2n}((vu)^m+(-1)^{nf+mb}(uv)^m)t^{2k}\quad \text{ for } a=0\\
\displaystyle\sum_{n\equiv k \ (\mathrm{mod} \ 2)}\alpha_{n,m,l}u^{2n}((vu)^m+(-1)^{n(f+1)+m(b+1)}(uv)^m)t^{2k}\quad \text{ for } a=1.
\end{cases}
\]

Using the formulae above we can prove the following

\begin{theorem}
\label{fixedrings89}
Let $H$ be \kashina{8} or \kashina{9}. Then $H$ acts inner-faithfully on the algebras $A=\frac{\kk\langle u,v\rangle}{(r)}[t;\sigma]$ with $A_1 \ = \pi(u,v) \oplus T$ where $\pi$ is an irreducible two-dimensional representation given in the table below, $T$ is a one-dimensional representation with $x.t = \pm \i t$, $y.t=t,$ and  $\sigma=\begin{bmatrix}\alpha&0\\0&-\alpha\end{bmatrix}$, with $\alpha \in \kk^\times$, and $r=uv\pm vu$ or $r=u^2\pm v^2$. The invariant rings for these actions are given in the table below.
\begin{center}
\begin{tabular}{|c|c|c|l| }
\hline
Kashina & $\pi$ & Relation & Invariant ring\\
\hline  
8& $\pi_1$ & $uv-vu$ & $\kk\!\!\left[ \begin{array}{ll}u^2v^2,u^4+v^4,uv(u^4-v^4),\\uv(u^2+v^2)t^2,(u^2-v^2)t^2,t^4\end{array}\right]$\\ \hline
8& $\pi_1$ & $uv+vu$ & $\kk[uv,u^4+v^4,(u^2-v^2)t^2,t^4]$\\ \hline
8& $\pi_1$ & $u^2-v^2$ & $\kk[u^4,uv-vu,u^2(uv+vu)t^2,t^4]$\\ \hline
8& $\pi_1$ & $u^2+v^2$ & $\kk[u^4,uv-vu,u^2t^2,t^4]$\\ \hline
8& $\pi_2$ & $uv-vu$ & $\kk[u^2v^2,u^2+v^2,uvt^2,t^4]$\\ \hline
8& $\pi_2$ & $uv+vu$ & $\kk[u^2v^2,u^2+v^2,uv(u^2-v^2)t^2,t^4]$\\ \hline
8& $\pi_2$ & $u^2-v^2$ & $\kk[u^2,(uv)^2+(vu)^2,(uv+vu)t^2,t^4]$\\ \hline
8& $\pi_2$ & $u^2+v^2$ &
$\kk\!\!\left[ \begin{array}{ll}u^4,(uv)^2+(vu)^2,u^2((uv)^2-(vu)^2),\\u^2(uv-vu)t^2,(uv+vu)t^2,t^4\end{array}\right]$\\ \hline
9& $\pi_1$ & $uv-vu$ & $\kk[uv,u^4+v^4,(u^2-v^2)t^2,t^4]$\\ \hline
9& $\pi_1$ & $uv+vu$ & $\kk\!\!\left[\begin{array}{ll}u^2v^2,u^4+v^4,uv(u^4-v^4),\\(u^2-v^2)t^2,uv(u^2+v^2)t^2,t^4\end{array}\right]$\\ \hline
9& $\pi_1$ & $u^2-v^2$ & $\kk[u^4,uv+vu,u^2(uv-vu)t^2,t^4]$\\ \hline
9& $\pi_1$ & $u^2+v^2$ & $\kk[u^4,uv+vu,u^2t^2,t^4]$\\ \hline
9& $\pi_2$ & $uv-vu$ & $\kk[u^2v^2,u^2+v^2,uv(u^2-v^2)t^2,t^4]$\\ \hline
9& $\pi_2$ & $uv+vu$ & $\kk[u^2v^2,u^2+v^2,uvt^2,t^4]$\\ \hline
9& $\pi_2$ & $u^2-v^2$ &
$\kk[u^2,(uv)^2+(vu)^2,(uv-vu)t^2,t^4]$\\ \hline
9& $\pi_2$ & $u^2+v^2$ &
$\kk\!\!\left[ \begin{array}{ll}u^4,(uv)^2+(vu)^2,u^2((uv)^2-(vu)^2),\\u^2(uv+vu)t^2,(uv-vu)t^2,t^4\end{array}\right]$\\ \hline
\end{tabular}
\end{center}
~\\
Hence, by Lemma~\ref{notregular}, \kashina{8} and \kashina{9} are not reflection Hopf algebras for any of these actions.
\end{theorem}

\begin{proof}
The calculation of these fixed rings is a combination of the results in  \cite[Subsection 1.4]{FKMW} and induction.

For the first row of the table (when $H$ is \kashina{8}, $A_1$ is the representation $\pi_1$, and $A$ is defined by the relation $uv-vu$), every invariant must be of the form
\[
\sum\alpha_{j,n,k}u^jv^j(u^{4n}+(-1)^jv^{4n})t^{4k}+\sum\beta_{j,n,k}u^jv^j(u^{4n+2}-(-1)^jv^{4n+2})t^{4k+2},
\]
which is the same as for the tenth row of the table (\kashina{9} , $\pi_1$, relation
$uv+vu$).  By \cite[Remark 1.12]{FKMW}
taking $x = u^4$, $y = -v^4$ and $z = uv$,
  the first sum in this invariant is generated by
  $u^2v^2,u^4+v^4,uv(u^4-v^4)$ and $t^4$.    
We claim that to generate the
second sum one only needs $(u^2-v^2)t^2,uv(u^2+v^2)t^2$. Indeed since
$u^2v^2$ and $t^4$ are among the generators one needs to show only
that the elements $(u^{4n+2}-v^{4n+2})t^2$ and
$uv(u^{4n+2}+v^{4n+2})t^2$ are generated. We induct on $n$. For the
first element it suffices to notice that
\[
(u^{4n}+v^{4n})(u^2-v^2)t^2=(u^{4n+2}-v^{4n+2})t^2-u^2v^2(u^{4n-2}-v^{4n-2})t^2,
\]
and for the second
\[
(u^{4n}+v^{4n})uv(u^2+v^2)t^2=uv(u^{4n+2}+v^{4n+2})t^2+u^3v^3(u^{4n-2}+v^{4n-2})t^2.
\]
For \kashina{8}, $\pi_1$, and  relation $uv+vu$ an invariant must be of the form 
\[
\sum\alpha_{j,n,k}u^jv^j(u^{4n}+v^{4n})t^{4k}+\sum\beta_{j,n,k}u^jv^j(u^{4n+2}-v^{4n+2})t^{4k+2},
\]
which is the same as for \kashina{9}, $\pi_1$, and relation $uv-vu$. By  \cite[Lemma 1.10 (1)]{FKMW} the first sum in this invariant is generated by $uv,u^4+v^4,t^4$. We claim that to generate the second sum one only needs $(u^2-v^2)t^2$. It suffices to show that the element $(u^{4n+2}-v^{4n+2})t^2$ can be generated. This follows by induction and from the equality
\[
(u^{4n}+v^{4n})uv(u^2-v^2)t^2=(u^{4n+2}-v^{4n+2})t^2-u^2v^2(u^{4n-2}-v^{4n-2})t^2.
\]
For \kashina{8}, $\pi_1$, and  relation $u^2-v^2$ an invariant must be of the form
\[
\sum\alpha_{l,j,n}u^{4l}((uv)^j+(-1)^j(vu)^j)t^{4n}+\sum\beta_{l,j,n}u^{4l+2}((uv)^j-(-1)^j(vu)^j)t^{4n+2}.
\]
By the results in  \cite[Subsection 1.4]{FKMW}, the first sum is generated by $u^4,uv-vu,t^4$. We claim that the invariant $u^2(uv+vu)t^2$ is sufficient to generate the second sum. It suffices to prove by induction that one can generate the element $u^2((uv)^j-(-1)^j(vu)^j)t^2$. This follows from the equality
\begin{align*}
((uv)^j+(-1)^j(vu)^j)u^2(uv+vu)t^2&=u^2((uv)^{j+1}-(-1)^{j+1}(vu)^{j+1})t^2+\\
&+u^6((uv)^{j-1}-(-1)^{j-1}(vu)^{j-1})t^2.
\end{align*}
All the remaining cases are proved similarly, and therefore we omit the proofs.
\end{proof}

\section{$\bG(H) = D_8$}\label{d8}
To study \twokashinas{12}{13}, which are the Hopf algebras $H$ with $\bG(H) = D_8$ (the dihedral group of order 8), we use the fact that they are both cotwists of group algebras: 
\[ H_{C:1} \cong (\kk D_{16})^J \quad \mbox{and} \quad  H_E \cong (\kk \textit{SD}_{16})^J
\]
where $D_{16} = \langle a, b \mid a^8 = b^2 = 1, ba = a\inv b \rangle$ and $\textit{SD}_{16} = \langle a, b \mid a^8 = b^2 = 1, ba = a^3b \rangle$ denotes the semidihedral (quasidihedral) group of order 16.
See \cite[pp. 658--659]{kashina}.

Both $D_{16}$ and $\textit{SD}_{16}$ have a subgroup $F$ isomorphic to the Klein-4 group, in both cases given by $\{1,a^4, b, a^4 b\}$. The specific dual cocycle $J \in F \otimes F$ is given in \cite[Section 7]{kashina}. Let $c = a^4$. Then
\begin{align*}
J &= \delta_1 \otimes \delta_1 + \delta_1 \otimes \delta_{c} + \delta_1 \otimes \delta_b + \delta_1 \otimes \delta_{c b} 
+ \delta_{c} \otimes \delta_1 + \delta_{c} \otimes \delta_{c} + \i \delta_{c} \otimes \delta_b - \i \delta_{c} \otimes \delta_{c b} \\
&+ \delta_{b} \otimes \delta_1 - \i \delta_{b} \otimes \delta_{c} + \delta_{b} \otimes \delta_b + \i \delta_{b} \otimes \delta_{c b} 
+ \delta_{c b} \otimes \delta_1 + \i \delta_{cb} \otimes \delta_{c} - \i \delta_{cb} \otimes \delta_b + \delta_{cb} \otimes \delta_{c b} 
\end{align*}
where
\begin{align*}
\delta_1 &= \frac{1}{4} ( 1 + c + b + cb) &
\delta_{c} &= \frac{1}{4} ( 1 + c - b - cb)\\
\delta_{b} &= \frac{1}{4} ( 1 - c + b - cb) &
\delta_{c b} &= \frac{1}{4} ( 1 - c - b + cb)
\end{align*}
\subsection{Kashina \#12}

There are three non-isomorphic irreducible two-dimensional representations of $D_{16}$, namely
\[ \pi_j(a) = \begin{bmatrix} \omega^{j} & 0 \\ 0 & \omega^{-j} \end{bmatrix} \quad \pi_j(b) = \begin{bmatrix}0 & 1 \\ 1 & 0 \end{bmatrix}
\]
where $j = 1,2,3$ and $\omega = e^{2 \pi \i / 8}$. There are four one-dimensional representations $T_{\pm 1, \pm 1}(t)$ where the subscript indicates the action on $a$ and $b$ respectively, on $t$.

If $\kk D_{16}$ acts on an algebra $A$ where $A_1=\pi_2$ then the two-sided ideal generated by $1-a^4$ is a Hopf ideal that annihilates this algebra, and therefore the action is not inner-faithful (one can check that $a^4$ is group-like). Both $\pi_1$ and $\pi_3$ are inner-faithful since
\[\pi_1(u,v) \otimes \pi_1(u,v) = \pi_3(u,v) \otimes \pi_3(u,v) =  \pi_2(u^2,v^2) \oplus T_{1,1}(uv+vu) \oplus T_{1,-1}(uv-vu),\]
\[
\pi_2(u,v)\otimes\pi_2(u,v)=T_{1,1}(uv+vu)\oplus T_{1,-1}(uv-vu)\oplus T_{-1,1}(u^2+v^2)\oplus T_{-1,-1}(u^2-v^2).
\]

Therefore, $\kk D_{16}$ acts inner-faithfully on $\kk[u,v]$ and $\kk_{-1}[u,v]$. Computations similar to the ones performed for the previous Hopf algebras in this paper show that $\kk[u,v]^{\kk D_{16}}=\kk[uv,u^8+v^8]$, which is AS regular since, by the Chevalley--Shephard--Todd Theorem, $D_{16}$ is a reflection group for $\kk[u,v]$. On the other hand, $\kk_{-1}[u,v]^{\kk D_{16}}=\kk[u^2v^2,u^8+v^8,uv(u^8-v^8)]$ is not AS regular since, by \cite[Theorem 2.1]{CKWZ}, $\kk D_{16}$ acts with trivial homological determinant on $\kk_{-1}[u,v]$ so, by \cite[Theorem 2.3]{CKWZ}, the fixed ring is not AS regular.

By Lemma~\ref{lem.cocycle}, the only two-dimensional AS regular algebras that \kashina{12} acts on inner-faithfully are $\kk[u,v]_J$ and $\kk_{-1}[u,v]_J$. By the same lemma it follows that the fixed rings are the same.

Let $A=\kk[u,v]_J$ with $A_1 = \pi_1$ or $\pi_3$. Let $*$ denote the multiplication in $A$. Since $c=a^4$ acts as $-1$ and $b$ swaps $u$ and $v$, we have that
\begin{align*}\delta_1 . u &= 0 &  \delta_1 .v &= 0\\
\delta_c.u &= 0 &  \delta_c.v &= 0 \\
\delta_b.u &= \frac{1}{2}(u+v) & \delta_b.v &= \frac{1}{2}(u+v)\\
\delta_{cb}.u &= \frac{1}{2}(u-v) & \delta_{cb}.v &= -\frac{1}{2}(u-v)
\end{align*}
and therefore
\begin{align*}u*u &=  \sum (J^{(1)} . u)(J^{(2)} .u) = \frac{1}{2}(u^2+v^2) &
 u*v &= uv - \frac{\i}{2}(u^2-v^2)\\
 v*u &= uv + \frac{\i}{2}(u^2-v^2) &
 v*v &= \frac{1}{2}(u^2+v^2).\end{align*}
 Using these equations,  with $X = \frac{1}{2}(u + v)$ and $Y = \frac{1}{2}(u - v)$,
one can check that $XY + YX = 0$.  Since  taking a cocycle twist of a graded algebra preserves the Hilbert series,
there can  be only a single quadratic relation in $A$.  It follows that $A \cong \kk_{-1}[X,Y]$.  Then 
\[  uv = \frac{1}{4}\left((u+v)*(u+v) - (u-v)*(u-v)\right) = X^2 - Y^2,
\]
one can also directly check that 
\[ u^8 + v^8 = (X^2-Y^2)^4+32X^2Y^2(X^2+Y^2)^2.
\]

Let $B=\kk_{-1}[u,v]_J$, with $A_1=\pi_1$ or $\pi_3$. Let $*$ denote the multiplication in $B$. Then one can check that
\begin{align*} u*u &= \frac{1}{2}(u^2+v^2)-\i uv &
 u*v &= -\frac{\i}{2}(u^2-v^2)\\
 v*u &= \frac{\i}{2}(u^2-v^2) &
 v*v &=\frac{1}{2}(u^2+v^2)+\i uv.
\end{align*}
By setting $X=\frac{1}{2}(u+v)$ and $Y=\frac{1}{2}(u-v)$ it follows that $B\cong \frac{\kk\langle X,Y\rangle}{(X^2-Y^2)}$. Direct computations show that 
\begin{align*} 
u^2v^2&=2X^4+(XY)^2+(YX)^2,\\
u^8+v^8&=140X^8-56X^5(YX)Y-56X^4(YX)^2+2X(YX)^3Y+2(YX)^4,\\
uv(u^8-v^8)&=84\i(X^9Y+X^8(YX))-54\i(X^5(YX)^2Y+X^4(YX)^3)\\&\;\;\;+2\i(XY)^4Y+(YX)^5)
\end{align*}
We notice that
\[
u^8+v^8+14(u^2v^2)^2=16(14X^8+(XY)^4+(YX)^4)
\]

From the previous results we can deduce the following
\begin{theorem} \label{fixedring12}
Let $H$ be \kashina{12}. Then $H$ acts inner-faithfully on the algebra $\frac{\kk\langle X,Y\rangle}{(r)}$ with $r=XY+YX$ or $r=X^2-Y^2$, and
\begin{align*}
a.X=\frac{\sqrt2}{2}(\pm X+\i Y),\quad b.X=X\\
a.Y=\frac{\sqrt2}{2}(\i X\pm Y),\quad b.Y=-Y.
\end{align*}
The invariant rings are 
\begin{center}
\begin{tabular}{|c|l| }
\hline
Relation & Invariant ring\\\hline  
$XY+YX$& $\kk[X^2-Y^2,X^2Y^2(X^2+Y^2)^2]$ \\\hline
$X^2-Y^2$ & $\kk\!\!\left[ \begin{array}{ll}2X^4+(XY)^2+(YX)^2,\\14X^8+(XY)^4+(YX)^4,\\42X^8(XY+YX)-27X^4((XY)^3+(YX)^3)\\+((XY)^5+(YX)^5)\end{array}\right]$\\\hline
\end{tabular}
\end{center}
The first invariant ring is a commutative polynomial ring, hence AS regular, while by Lemma~\ref{notregular}, the second is not AS regular.
\end{theorem}

\subsection{Kashina \#13}

There are three non-isomorphic irreducible two-dimensional representations of $\textit{SD}_{16}$, namely
\[ \pi_1(a) = \begin{bmatrix} \omega & 0 \\ 0 & \omega^{3} \end{bmatrix} \quad \pi_1(b) = \begin{bmatrix}0 & 1 \\ 1 & 0 \end{bmatrix}
\]
\[ \pi_2(a) = \begin{bmatrix} \i & 0 \\ 0 & -\i \end{bmatrix} \quad \pi_2(b) = \begin{bmatrix}0 & 1 \\ 1 & 0 \end{bmatrix}
\]
\[ \pi_3(a) = \begin{bmatrix} \omega\inv & 0 \\ 0 & \omega^{-3} \end{bmatrix} \quad \pi_3(b) = \begin{bmatrix}0 & 1 \\ 1 & 0 \end{bmatrix}
\]
where $\omega = e^{2 \pi \i / 8}$. There are four one-dimensional representations $T_{\pm 1, \pm 1}$ where the subscript indicates the action on $a$ and $b$ respectively. Arguing as for \kashina{12} we deduce that both $\pi_1$ and $\pi_3$ are faithful group representations of $\textit{SD}_{16}$ while $\pi_2$ is not.
Hence, $\kk \textit{SD}_{16}$ acts inner-faithfully on two generated algebras $A$ where $A_1 = \pi_1$ or $\pi_3$. We have that
\[\pi_1(u,v) \otimes \pi_1(u,v) = \pi_2(u^2,v^2) \oplus T_{-1,1}(uv+vu) \oplus T_{-1,-1}(uv-vu)\]
\[ \pi_3(u,v) \otimes \pi_3(u,v) = \pi_2(v^2,u^2) \oplus T_{-1,1}(uv+vu) \oplus T_{-1,-1}(uv-vu).
\]
Note that $\pi_1(u,v) \otimes \pi_3(u,v) = \pi_2(vu,uv) \oplus T_{1,1}(u^2+v^2) \oplus T_{1,-1}(u^2-v^2)$ so $\pi_1$ and $\pi_3$ are not isomorphic.

Therefore the two generated AS regular algebras that $\kk \textit{SD}_{16}$ acts inner-faithfully on are $\kk[u,v]$ and $\kk_{-1}[u,v]$. 

We first compute $\kk[u,v]^{\kk \textit{SD}_{16}}$. Let $F=\sum\alpha_{n,m}u^nv^m$ be an invariant. If $b.F=F$ then $\alpha_{n,m}=\alpha_{m,n}$. If $a.F=F$ then $n\equiv 5m\pmod{8}$. Therefore an invariant must be of the form
\[
F=\sum_{5m+8k\geq0}\alpha_{k,m}(u^{5m+8k}v^m+u^mv^{5m+8k}).
\]
We show that $F$ is generated by the elements $u^2v^2,uv(u^4+v^4),u^8+v^8$. If $4m+8k\geq0$ then we can write $F$ as
\[
F=\sum_{5m+8k\geq0}\alpha_{k,m}u^mv^m(u^{4m+8k}+v^{4m+8k}),
\]
if $m=2h$ for some $h$ then $u^mv^m$ can be generated by $u^2v^2$, while $(u^8)^{h+k}+(v^8)^{h+k}$ can be generated by $u^8+v^8$ and $u^8v^8$ by \cite[Lemma 1.10(1)]{FKMW}. If $m=2h+1$ then it suffices to generate $uv(u^{8(h+k)+4}+v^{8(h+k)+4})$. Indeed
\begin{align*}
uv(u^4+v^4)(u^{8(h+k)}+v^{8(h+k)})&=uv(u^4v^{8(h+k)}+v^4u^{8(h+k)})\\
&+uv(u^{8(h+k)+4}+v^{8(h+k)+4}),
\end{align*}
if $h+k=0$ then we are done, otherwise we write the first summand as $$u^4v^4uv(u^{8(h+k)-4}+v^{8(h+k)-4})$$ which we can generate by induction. If $4m+8k\leq0$ then a similar argument shows that 
\[
F=\sum_{5m+8k\geq0}\alpha_{k,m}u^{5m+8k}v^{5m+8k}(u^{-4m-8k}+v^{-4m-8k})
\]
can be generated by $u^2v^2,uv(u^4+v^4),u^8+v^8$. This fixed ring is not AS regular since $\textit{SD}_{16}$ is not a classical reflection group.

An analogous proof shows that $\kk_{-1}[u,v]^{\kk \textit{SD}_{16}}=\kk[u^2v^2,uv(u^4-v^4),u^8+v^8]$, which is not AS regular.

Using the same change of variables done for \kashina{12} one can check that in $\kk[u,v]$
\begin{align*}
u^2v^2&=(X^2-Y^2)^2,\\
uv(u^4+v^4)&=2(X^6-Y^6+5X^2Y^2(X^2-Y^2)),\\
u^8+v^8&=(X^2-Y^2)^4+32X^2Y^2(X^2+Y^2)^2.
\end{align*}
While in $\kk_{-1}[u,v]$ one has
\begin{align*}
u^2v^2&=2X^4+(XY)^2+(YX)^2,\\
uv(u^4-v^4)&=8X^2((XY)^2-(YX)^2),\\
u^8+v^8&=140X^8-56X^5(YX)Y-56X^4(YX)^2+2X(YX)^3Y+2(YX)^4.
\end{align*}
From the previous results we can deduce the following
\begin{theorem} \label{fixedring13}
Let $H$ be \kashina{13}. Then $H$ acts inner-faithfully on the algebra $\frac{\kk\langle X,Y\rangle}{(r)}$ with $r=XY+YX$ or $r=X^2-Y^2$, and
\begin{align*}
a.X=\pm\i Y,\quad b.X=X\\
a.Y=\pm\i X,\quad b.Y=-Y.
\end{align*}
The invariant rings are 
\begin{center}
\begin{tabular}{|c|l| }
\hline
Relation & Invariant ring\\\hline  
$XY+YX$& $\kk\!\!\left[\begin{array}{ll}(X^2-Y^2)^2,X^2Y^2(X^2+Y^2)^2\\X^6-Y^6+5X^2Y^2(X^2-Y^2)\end{array}\right]$ \\\hline
$X^2-Y^2$ & $\kk\!\!\left[ \begin{array}{ll}2X^4+(XY)^2+(YX)^2,\\14X^8+(XY)^4+(YX)^4,\\X^2((XY)^2-(YX)^2)\end{array}\right]$\\\hline
\end{tabular}
\end{center}
By Lemma~\ref{notregular}, neither of the invariant rings is AS regular.
\end{theorem}

\section{$\bG(H) = C_2 \times C_2$} \label{c2Xc2}
Kashinas \#14-16 have $\bG(H) = C_2 \times C_2$. \kashina 14 and \kashina 16 were
considered in \cite{FKMW}, and we include those results at the end of this section.

\kashina{15} (see \cite[pp. 659--660]{kashina}) is a smash coproduct $\kk Q_8 \#^{\alpha} \kk C_2$ where
$Q_8 = \langle a, b \mid a^4 = 1, b^2 = a^2, ba = a\inv b \rangle$ and $C_2 = \{1,g\}$.
The map $\alpha: C_2 \to \Aut(\kk Q_8)$ where $\alpha_g = \alpha(g)$ is defined by $\alpha_g(a) = b$ and $\alpha_g(b) = a$.

Since the algebra structure on $H$ is simply that of $\kk Q_8 \otimes \kk C_2$,  we may present $H$ as generated
by $a,b,g$ subject to the relations
$$a^4 = 1, b^2 = a^2, ba = a^3b, gb = ag, ga = bg.$$
To describe the rest of the Hopf algebra structure, let $\delta_1 = (1+g)/2$ and $\delta_g = (1-g)/2$.
The comultiplication, antipode, and counit are given by the following (using Sweedler notation):
\begin{align*}
\Delta (x\delta_{g^k}) &= \sum_{r+t = k}(x_{(1)}\delta_{g^r}) \otimes (\alpha_{g^r}(x_{(2)})\delta_{g^t}), \\
S(x\delta_{g^k}) &= \alpha_{g^k}(S(x))\delta_{g^{-k}}, \\
\epsilon (x\delta_{1}) &= \epsilon(x), \\
\epsilon (x\delta_{g}) &= 0
\end{align*}
for all $x \in \kk Q_8$. One can then check that the coproduct of \kashina{15} satisfies:
\begin{align*}
 \Delta(a) &= \frac{1}{2}(a \otimes a + ag \otimes a + a \otimes b - ag \otimes b) \\
 \Delta(b) &= \frac{1}{2}(b \otimes b + bg \otimes b + b \otimes a - bg \otimes a)  \\
 \Delta(g) &= g \otimes g.
\end{align*}
The group of grouplikes is $\{1,g,a^2,a^2g\}$.  There are eight
one-dimensional irreducible representations given by
$T_{\pm 1, \pm 1, \pm 1}$ where the subscript indicates the action on
$a$, $b$, and $g$ respectively. In the Grothendieck ring, we have
\begin{equation} \label{eqn:K15oneDimlTensor}
\begin{aligned}
T_{\alpha, \beta, 1}  \otimes T_{\alpha', \beta', \gamma} &  = T_{\alpha \alpha', \beta \beta', \gamma} \\
T_{\alpha, \beta, -1} \otimes T_{\alpha', \beta', \gamma} &  = T_{\alpha \beta', \alpha'\beta, -\gamma}.
\end{aligned}
\end{equation}

There are two non-isomorphic two-dimensional irreducible representations, given by
the following matrices, for $i \in \{1,2\}$:
\[
\pi_i(a) = \begin{bmatrix} 0 & \i \\ \i & 0 \end{bmatrix} \quad  \pi_i(b) = \begin{bmatrix} 0 & -1 \\ 1 & 0 \end{bmatrix} \quad \pi_i(g) = \begin{bmatrix} (-1)^{i+1} & 0 \\ 0 & (-1)^{i+1}\end{bmatrix}.
\]
The action of $H$ on $\pi_1(u,v)^{\otimes 2} \oplus \pi_2(u,v)^{\otimes 2}$ is:
\begin{center}
\begin{tabular}{|c||c|c|c|c||c|c|c|c|}
\hline
& \multicolumn{4}{c||}{$\pi_1^{\otimes 2}$} & \multicolumn{4}{c|}{$\pi_2^{\otimes 2}$} \\ \cline{2-9}
& $u^2$ & $v^2$ & $uv$ & $vu$ & $u^2$ & $v^2$ & $uv$ & $vu$ \\ \hline
$a$ & $-v^2$& $-u^2$ & $-vu$ & $-uv$ & $\i v^2$& $- \i u^2$ & $-\i vu$ & $\i uv$\\ \hline
$b$ & $v^2$ & $u^2$ & $-vu$ & $-uv$ & $\i v^2$ & $-\i u^2$ & $\i vu$ & $-\i uv$\\ \hline
$g$ & $u^2$ & $v^2$ & $uv$& $vu$ & $u^2$ & $v^2$ & $uv$& $vu$\\
\hline
\end{tabular}
\end{center}
Therefore one has decompositions:
\begin{align*}
\pi_1(u,v) \otimes \pi_1(u,v) =& ~T_{1,1,1}(uv-vu) \oplus T_{1,-1,1}(u^2-v^2) \oplus \\
                               & ~T_{-1,1,1}(u^2+v^2) \oplus T_{-1,-1,1}(uv+vu) \\
\pi_2(u,v) \otimes \pi_2(u,v) =& ~T_{1,1,1}(u^2 + \i v^2) \oplus T_{1,-1,1}(uv - \i vu) \oplus \\
                               & ~T_{-1,1,1}(uv + \i vu) \oplus T_{-1,-1,1}(u^2 - \i v^2)
\end{align*}
One may also check that $\pi_i \otimes T_{\beta,\gamma,1} \cong \pi_i$ for $i = 1,2$ and
$\beta,\gamma \in \{\pm 1\}$.  It follows that \kashina{15} does not act inner-faithfully on any
two-generated algebras.  

The action of \kashina{15} on the representation
\begin{align*}
  (\pi_1(u,v) \otimes T_{\beta,\gamma,-1}) \oplus (T_{\beta,\gamma,-1} \otimes \pi_1(u,v)) \oplus \\
  (\pi_2(u,v) \otimes T_{\beta,\gamma,-1}) \oplus (T_{\beta,\gamma,-1} \otimes \pi_2(u,v))
\end{align*}
is given in the table below:
\begin{center}
\begin{tabular}{|c||c|c|c|c||c|c|c|c||}
\hline
 & \multicolumn{4}{c||}{$(\pi_1 \otimes T) \oplus (T \otimes \pi_1)$} 
 & \multicolumn{4}{c||}{$(\pi_2 \otimes T) \oplus (T \otimes \pi_2)$} \\ \cline{2-9}
    & $ut$         & $vt$          & $tu$        & $tv$        & $ut$        & $vt$        & $tu$        & $tv$   \\ \hline
$a$&$\beta\i vt$&$\beta\i ut$&$\beta tv$&$-\beta tu$&$\gamma\i vt$&$\gamma\i ut$&$\beta tv$&$-\beta tu$ \\ \hline
$b$&$\gamma vt$&$-\gamma ut$&$\gamma\i tv$&$\gamma\i tu$&$\beta vt$&$-\beta ut$&$\gamma\i tv$&$\gamma\i tu$\\ \hline
$g$ & $-ut$        & $-vt$         & $-tu$        & $-tv$          & $ut$         & $vt$         & $tu$    & $tv$ \\
\hline
\end{tabular}
\end{center}


One therefore has the following decompositions:
\begin{align*}
\pi_1(u,v) \otimes T_{\beta,\gamma,-1} &=
\begin{cases}
\pi_2(ut,\beta vt) & \beta = \gamma \\
\pi_2(vt,\beta ut) & \beta \neq \gamma
\end{cases} \\
T_{\beta,\gamma,-1} \otimes \pi_1(u,v) &= 
\begin{cases}
\pi_2(tv,\beta\i tu) & \beta = \gamma \\
\pi_2(tu,-\beta\i tv) & \beta \neq \gamma
\end{cases} \\
\pi_2(u,v) \otimes T_{\beta,\gamma,-1} &=
\begin{cases}
\pi_1(ut,\beta vt) & \beta = \gamma \\
\pi_1(vt,-\beta ut) & \beta \neq \gamma
\end{cases} \\
T_{\beta,\gamma,-1} \otimes \pi_2(u,v) &= 
\begin{cases}
\pi_1(tv,\beta\i tu) & \beta = \gamma \\
\pi_1(tu,-\beta\i tv) & \beta \neq \gamma.
\end{cases}
\end{align*}
 It follows that \kashina{15} acts inner-faithfully on $A = \frac{\kk \langle u, v \rangle}{(r)}[t;\sigma]$, where
$\sigma = \begin{bmatrix} 0 & \i  \alpha  \\ \alpha & 0 \end{bmatrix}$ for a nonzero scalar $\alpha$,
$r = uv \pm vu$ and $A_1$ is the representation $\pi_1(u,v) \oplus T_{\beta,\gamma,-1}$.
Note that $\sigma$ does not induce an automorphism on the base of the Ore extension
when $r = u^2 \pm v^2$. 

\kashina{15} also acts inner-faithfully on $A = \frac{\kk \langle u, v \rangle}{(r)}[t;\sigma]$, where
$\sigma = \begin{bmatrix} 0 & \beta\gamma\alpha\i  \\ \alpha & 0 \end{bmatrix}$ for a nonzero scalar $\alpha$,
$r = u^2 \pm \i v^2$ and $A_1$ is the representation $\pi_2(u,v) \oplus T_{\beta,\gamma,-1}(t)$.
Note that $\sigma$ does not induce an automorphism on the base of the Ore extension
when $r = uv \pm \i vu$. 

\begin{lemma} \label{lem:K15evenTdeg}
Suppose $H$ is \kashina{15} and $H$ acts inner-faithfully on $A$ where $A_1 = \pi_i(u,v) + T(t)$ where $i \in \{1,2\}$ and $T$ is a one-dimensional representation. Then the $t$-degree of any invariant must be even.
\end{lemma}

\begin{proof}
  If $i = 1$, then the claim is immediate from the action of $g$.
  Suppose therefore that $i = 2$ and 
  let $ft^h$ be a nonzero invariant with $f$ in the
  subalgebra generated by $u$ and $v$ and $h$ odd.
  
  Suppose that $\kk f$ is a representation of $H$.
    Then $\kk f$ must be the multiplicative inverse of the
    representation $\kk t^h$ in the Grothendieck ring.  No such
  nonzero $f$ exists, however, since $g$ acts as $-1$ on $\kk t^h$,
  and all one-dimensional representations that appear as a summand of
  a tensor power of $\pi_2$ have trivial action of $g$ (cf. Equation
  \eqref{eqn:K15oneDimlTensor}).

   It remains to prove that $\kk f$ is a
    representation of $H$.  Note that in the case under consideration,
    $g.f = (-1)^{\deg f + h}$, so that $h$ and $\deg f$ must have the
    same parity.  Therefore, we may assume that $\deg f$ is odd.

  We may now show that $\kk f$ is a representation of $H$.  Indeed,
  \begin{eqnarray*}
    ft^h = a.(ft^h) & = & \frac{1}{2}((a.f)(a.t^h) + (ag.f)(a.t^h) + (a.f).(b.t^h) - (ag.f)(b.t^h) \\
                    & = & \frac{1}{2}((a.f)(a.t^h) - (a.f)(a.t^h) + (a.f).(b.t^h) + (a.f)(b.t^h) \\
                    & = & (a.f)(b.t^h).
  \end{eqnarray*}
  Since $A$ is a domain, it follows that $a.f$ is a scalar multiple of $f$.  One
  may similarly show that $b.f$ is a scalar multiple of $f$, completing the proof.
\end{proof}

We now compute the invariants for the action of $H$ on $A=\frac{\kk\langle u,v\rangle}{(uv\pm vu)}[t;\sigma]$ with $A_1=\pi_1(u,v)\oplus T_{(-1)^p,(-1)^q,-1}$. Since $g$ acts trivially on $u,v,t^2$ it follows that $a,b$ act as group-like elements on $u,v,t^2$, therefore
\[
a.(u^nv^mt^{2h})=\i^{n+m+2(p+q)h}v^nu^mt^{2h},\quad b.(u^nv^mt^{2h})=(-1)^{m+(p+q)h}v^nu^mt^{2h}.
\]
Let $F=\sum\alpha_{n,m,h}u^nv^mt^{2h}$ be an invariant (the $t$-degree of $F$ must be even by Lemma~\ref{lem:K15evenTdeg}). The actions of $a$ and $b$ imply that if $\alpha_{n,m,h}\neq0$ then $\i^{n+m+2(p+q)h}=(-1)^{m+(p+q)h}$, which implies that $n\equiv m\pmod{4}$. Now arguing as in the previous Kashinas it follows that if $r=uv-vu$ then an invariant must be of the form
\[
F=\sum\alpha_{n,m,h}u^nv^n(u^{4m}+(-1)^{m+(p+q)h}v^{4m})t^{2h},
\]
and if $r=uv+vu$ then
\[
F=\sum\alpha_{n,m,h}u^nv^n(u^{4m}+(-1)^{(p+q)h}v^{4m})t^{2h}.
\]

We now compute the invariants for the action of $H$ on $A=\frac{\kk\langle u,v\rangle}{(u^2\pm\i v^2)}[t;\sigma]$ with $A_1=\pi_2(u,v)\oplus T_{(-1)^p,(-1)^q,-1}$. An induction argument shows
\[
a.(u^n)=\i^{n-\frac{n(n-1)(2n-1)}{6}}v^n,\quad b.(u^n)=\i^{\frac{n(n-1)(2n-1)}{6}}v^n,
\]
\[
a.(vu)^n=\i^n(uv)^n,\quad b.(vu)^n=\i^{3n}(uv)^n,
\]
\[
a.(uv)^n=\i^{3n}(vu)^n,\quad b.(uv)^n=\i^n(vu)^n,
\]
\[
a.(t^n)=(-1)^{(p+q)\frac{n(n-1)}{2}+pn}t^n,\quad b.(t^n)=(-1)^{(p+q)\frac{n(n-1)}{2}+qn}t^n.
\]
By Lemma~\ref{lem:K15evenTdeg} the degree of $t$ in an invariant must be even, therefore an invariant must be of the form
\[
F=\sum\alpha_{n,m,h}u^n(vu)^mt^{2h}+\sum\beta_{l,j,k}u^l(vu)^jvt^{2k},
\]
the action of $g$ implies that $n$ must be even and $l$ must be odd, therefore replacing $n$ with $2n$ and $l$ with $2l+1$ the invariant $F$ can be written as
\[
F=\sum\alpha_{n,m,h}u^{2n}(vu)^mt^{2h}+\sum\beta_{l,j,k}u^{2l}(uv)^{j+1}t^{2k}.
\]
Since $g$ acts trivially on monomials of even degree it follows that $a$ and $b$ act group-like on products of monomials of even degree, therefore using the formulae above it follows that
\[
a.(u^{2n}(vu)^mt^{2h})=\i^{-n+2n^2+m+2(p+q)h}v^{2n}(uv)^mt^{2h},
\]
\[
b.(u^{2n}(vu)^mt^{2h})=\i^{-n+2n^2+3m+2(p+q)h}v^{2n}(uv)^mt^{2h}.
\]
By equating the exponent of the $\i's$ modulo 4 it follows that $m$ must be even. Similarly it follows that $j$ must be odd. Now arguing as for the previous Kashinas it follows that if $r=u^2-\i v^2$ then an invariant must be of the form
\[
F=\sum\alpha_{n,m,h}u^{2n}((uv)^{2m}+(-1)^{m+(p+q)h}(vu)^{2m})t^{2h},
\]
and if $r=u^2+\i v^2$ then
\[
F=\sum\alpha_{n,m,h}u^{2n}((uv)^{2m}+(-1)^{n+m+(p+q)h}(vu)^{2m})t^{2h}.
\]

Using the results from  \cite[Subsection 1.4]{FKMW} we have the following

\begin{theorem} 
\label{fixedring15}
Let $H$ be \kashina{15}. Then $H$ acts inner-faithfully on the algebras $A$ and $B$ given below:

\begin{enumerate}
\item $A=\frac{\kk\langle u,v,\rangle}{(uv\pm vu)}[t;\sigma]$ with $A_1 = \pi(u,v) \oplus T(t)$ where $\pi$ is an irreducible two-dimensional representation with $g.u=u$, $T$ is a one-dimensional representation with $g.t=-t$, and $\sigma = \begin{bmatrix} 0 & \i  \alpha  \\ \alpha & 0 \end{bmatrix}$ for $\alpha \in \kk^\times$.
\item $B=\frac{\kk\langle u,v,\rangle}{(u^2\pm\i v^2)}[t;\sigma]$ with $B_1 = \pi(u,v) \oplus T(t)$ where $\pi$ is an irreducible two-dimensional representation with $g.u=-u$, $T$ is a one-dimensional representation with $a.t=(-1)^pt,  b.t=(-1)^qt,g.t=-t$ for $p,q \in \{0, 1\}$, and   $\sigma = \begin{bmatrix} 0 & (-1)^{p+q}\i  \alpha  \\ \alpha & 0 \end{bmatrix}$ with $\alpha \in \kk^\times$.
\end{enumerate}
The invariant rings for these actions are:

\begin{center}
\begin{tabular}{|c|c|c|l| }
\hline
$\pi$ & $r$ & $T_{(-1)^p,(-1)^q,-1}$ &Invariant ring\\
\hline  
$\pi_1$& $uv-vu$ & $p\equiv q\pmod{2}$ & $\kk[u^2v^2,u^4+v^4,uv(u^4-v^4),t^2]$\\ \hline
$\pi_1$& $uv-vu$ & $p\not\equiv q\pmod{2}$ & $\kk\!\!\left[\begin{array}{ll}u^2v^2,u^4+v^4,uv(u^4-v^4),\\(u^4-v^4)t^2,uvt^2,t^4\end{array}\right]$\\ \hline
 $\pi_1$& $uv+vu$ & $p\equiv q\pmod{2}$ & $\kk[uv,u^4+v^4,t^2]$\\ \hline
$\pi_1$& $uv+vu$ & $p\not\equiv q\pmod{2}$ & $\kk[uv,u^4+v^4,(u^4-v^4)t^2,t^4]$\\ \hline
$\pi_2$ &$u^2-\i v^2$ & $p\equiv q\pmod{2}$ & $\kk[u^2, (uv)^2-(vu)^2, t^2]$\\
\hline
$\pi_2$ & $u^2-\i v^2$ & $p \not\equiv q \pmod{2}$& $\kk[u^2,(uv)^2-(vu)^2, ((uv)^2+(vu)^2)t^2,t^4]$\\
\hline
$\pi_2$ & $u^2+ \i v^2$ & $p \equiv q\pmod{2}$ & $\kk[u^4,(uv)^2-(vu)^2, u^2((uv)^2+(vu)^2),t^2]$\\
\hline
 $\pi_2$& $u^2 + \i v^2$ & $p \not\equiv q \pmod{2}$& $\kk\!\!\left[\begin{array}{ll}u^4, (uv)^2-(vu)^2, u^2((uv)^2+(vu)^2),\\ u^2t^2, ((uv)^2 +(vu)^2)t^2, t^4\end{array}\right] $\\
\hline
\end{tabular}
\end{center}
The invariant rings in the third and fifth rows are Ore extensions of commutative polynomial rings, and hence \kashina{15} is a reflection Hopf algebra for the actions described in those rows.  By Lemma~\ref{notregular}, the other invariant rings are not AS regular.
\end{theorem}

\kashina{14} (\cite[p. 638]{kashina}) is $\mathcal{A}_{16}$ and \kashina 16 (\cite[pp. 640--641]{kashina})is $\mathcal{B}_{16}$, in the notation of \cite{M} and \cite{FKMW}. See \cite[Section 6]{FKMW} for the details on $\mathcal{A}_{16}$ and \cite[Section 4]{FKMW} for the details on $\mathcal{B}_{16}$.

\begin{theorem}(\cite[Section 6.3, and  Theorem 6.10]{FKMW}) \label{fixedring14}
$H=\cA_{16}$ has no inner-faithful two-dimensional representations, but acts inner-faithfully on five AS regular Ore extensions of dimension 3, and is a reflection Hopf algebra for three of these algebras.
\end{theorem}
\begin{theorem}(\cite[Section 4.2,  Theorems 4.7 and 4.8]{FKMW}) \label{fixedring16}
$H=\cB_{16}$ has eight inner-faithful two-dimensional representations,  and is a reflection Hopf algebra for four of these algebras.
\end{theorem}

\bibliographystyle{amsplain}
\bibliography{biblio}
\end{document}